\documentclass[11pt]{article}
 \usepackage{multirow}
\usepackage{graphicx}
\usepackage{subfigure}
\usepackage{epstopdf}
\usepackage{mathrsfs}
\usepackage{bbm}
\usepackage{booktabs}
\usepackage{bm}
\usepackage{float}
\usepackage{amsfonts}
\usepackage{amssymb}
\usepackage{amsmath,amsthm}
\usepackage{cases}
\usepackage{indentfirst}
\usepackage{hyperref}
\usepackage{booktabs}
\usepackage{array}
\usepackage{algorithm}
\usepackage{algorithmic}
\usepackage{color}
\usepackage{stfloats}
\usepackage[marginal]{footmisc}
\usepackage{enumitem}
\usepackage[numbers]{natbib}
\numberwithin{equation}{section}
\def\opn#1#2{\def#1{\operatorname{#2}}} 
\opn\chara{char} \opn\length{\ell}
\allowdisplaybreaks
 \opn\projdim{proj\,dim} \opn\injdim{inj\,dim}
\opn\rank{rank} \opn\depth{depth} \opn\grade{grade}
\opn\height{height} \opn\embdim{emb\,dim} \opn\codim{codim}

\opn\Tr{Tr} \opn\bigrank{big\,rank}
\opn\superheight{superheight}\opn\lcm{lcm}
\opn\trdeg{tr\,deg}%
\opn\reg{reg} \opn\lreg{lreg}
\opn\Ker{Ker} \opn\Coker{Coker} \opn\Im{Im} \opn\Hom{Hom}
\opn\Tor{Tor} \opn\Ext{Ext} \opn\End{End} \opn\Aut{Aut} \opn\id{id}

\opn\nat{nat}
\opn\pff{pf}
\opn\Pf{Pf} \opn\GL{GL} \opn\SL{SL} \opn\mod{mod} \opn\ord{ord}

\def\Implies{\ifmmode\Longrightarrow \else
     \unskip${}\Longrightarrow{}$\ignorespaces\fi}
\def\implies{\ifmmode\Rightarrow \else
     \unskip${}\Rightarrow{}$\ignorespaces\fi}
\def\iff{\ifmmode\Longleftrightarrow \else
     \unskip${}\Longleftrightarrow{}$\ignorespaces\fi}

\let\:=\colon

\newcommand{\p}{\partial}

\newtheorem{Theorem}{Theorem}[section]

\newtheorem{Lemma}[Theorem]{Lemma}

\newtheorem{Remark}[Theorem]{Remark}

\newtheorem{example}{Example}[section]

\let\epsilon=\varepsilon
\let\kappa=\varkappa
\usepackage{geometry}
\opn\ini{in} \opn\inm{inm} \opn\Sym{Sym} \opn\diag{diag}
\opn\Ii{(i)} \opn\Iii{(ii)}

\begin{document}
\title{On spectral Petrov-Galerkin method for solving optimal control problem governed by a two-sided fractional diffusion equation }


\author{Shengyue Li$^a$, Wanrong Cao$^{a,*}$, Yibo Wang$^a$}
\date{}
\date{\footnotetext{$^*$Corresponding author: wrcao@seu.edu.cn.}
	\footnotetext{$^a$ School of Mathematics, Southeast University, Nanjing, 210096.}}

\maketitle
{\bf\small Abstract}:\ {\small} In this paper, we investigate a spectral Petrov-Galerkin method for an optimal control problem governed by a two-sided space-fractional diffusion-advection-reaction equation. Taking into account the effect of singularities near the boundary generated by the weak singular kernel of the fractional operator, we establish the regularity of the problem in weighted Sobolev space. Error estimates are provided for the presented spectral Petrov-Galerkin method and the convergence orders of the state  and  control variables are determined. Furthermore, a fast projected gradient algorithm with a quasi-linear complexity is presented to solve the resulting discrete system.  Numerical experiments show the validity of theoretical findings and efficiency of the proposed fast algorithm.

\vspace{0.5em}

\vspace{0.5em}

\section{Introduction}

{Optimal control problems with partial differential equation constraints  have wide applications in science and engineering fields, such as biology, ecology,  economic, finance, etc., see \cite{athans,chen2005,geer,lions,hinze}. In recent two decades, fractional diffusion equations (FDEs) have  attracted increasing attention due to the capability of modeling complex physical phenomena with anomalous diffusion or long-time memory,  such as contaminant transport in ground water flow \cite{benson2000}, anomalous transport in biology \cite{hofl}, viscoelasticity \cite{main2010}, etc.  As a consequence,  both theory and numerical methods to optimal control problems governed by FDEs have gradually aroused more and more researchers'  interest.
Various numerical methods have been investigated for the optimal control problems with FDE constraints, e.g. finite element methods
\cite{antil2016, antil2017,du2018,gun2019, jin2020,  xie2020, zhou2015, zhang2019, zhou2019}, spectral Galerkin  methods \cite{wang2021,ye2016,ye2013,zhanglu}, collocation methods \cite{li2018,li2019,zaky2017}, etc. 

}

In this paper, we consider the following optimal control problem  governed by a space-fractional diffusion-advection-reaction equation:
	\begin{equation}\label{FOCPa1}
		  \min\limits_{q\in U_{ad}}J(u,q) = \frac{1}{2}\|u-u_d\|^2_{L^2(\Omega)} + \frac{\gamma}{2}\|q\|^2_{L^2(\Omega)}
	\end{equation}
	subject to
	\begin{eqnarray}\label{FOCPb1}\left\{\begin{aligned}
			\mathcal{L}_\theta^\alpha u  + \lambda_1 Du + \lambda_2 u&= f(x) + q(x),\ & x 	\in \Omega:=(0,1), \\
			u(0)&= 0,\ & x \in \partial \Omega,\quad\quad\quad\
	\end{aligned}\right.\end{eqnarray}
	where  $U_{ad}$ is an admissible set defined by
	\begin{equation}\label{admissibleset}
		U_{ad}=\{q\in L^{2}(\Omega):\int_{\Omega}q(x)dx\geq0 \},
	\end{equation}
$D$ denotes the first-order derivative with respect to $x$,  $f$ and $u_d$ are given functions, $\gamma$, $\lambda_1$ and $\lambda_2$ are constants, $\gamma>0$, $\lambda_1\neq0$, $\lambda_2\geq0$.  $\mathcal{L}_\theta^\alpha u:= -[\theta\
{}_{0}D_x^{\alpha}+(1-\theta)\ {}_{x}D_1^{\alpha}] $ is a general two-sided fractional operator with $\alpha\in (1,2)$ and $\theta\in [0,1]$.  Here ${}_{0}D_x^{\alpha}$ and ${}_{x}D_1^{\alpha}$ are left and right
Riemann-Liouville fractional derivatives \cite{samko1}, respectively, defined by
\begin{align*}
{}_{0}D_x^{\alpha}u(x)&=\frac{1}{\Gamma(2-\alpha)}\frac{d^2}{dx^2}\int_0^x\frac{u(s)}{(x-s)^{\alpha-1}}ds, \ x>0,\\
{}_{x}D_1^{\alpha}u(x)&=\frac{1}{\Gamma(2-\alpha)}\frac{d^2}{dx^2}\int_x^1\frac{u(s)}{(s-x)^{\alpha-1}}ds, \ x<1.
\end{align*}	

Due to the existence of  weak singular kernel of fractional derivatives, exact solutions of model equation \eqref{FOCPb1} and its variants usually exhibit singularities near the boundary. Even given smooth inputs, it leads to the low regularity of the problem in standard Sobolev space. In order to compensate for the weak singularities of solutions, weighted Sobolev spaces were employed in \cite{ervin2020,hao2020}. For the model equation \eqref{FOCPb1} with $\theta=1/2$, Hao and
Zhang gave sharp regularity estimates of solutions in weighted Sobolev space \cite{hao2020}. The work was extended by Ervin in \cite{ervin2020} to the general case, i.e. $\theta\in [0,1]$. Based on the regularity
results, a  spectral Petrov-Galerkin method with higher convergence order was obtained in \cite{zheng2020} for \eqref{FOCPb1}. Above works motivate us to analyze the regularity of the solution
to the optimal control problem with FDE constraints \eqref{FOCPa1}-\eqref{admissibleset} ,  and then present a spectral Petrov-Galerkin method and investigate its convergence order in weighted Sobolev space.

When taking  $\theta=1/2$, the operator {\small$\mathcal{L}_{1/2}^\alpha$} is equivalent to the \emph{integral} fractional Laplacian operator (sometimes called \emph{Riesz} fractional Laplacian). Zhang and Zhou \cite{zhanglu} considered a spectral Galerkin approximation to  optimal
control problem with \emph{integral} fractional Laplacian equation constraints. They derived the first-order optimality condition and studied the regularity of solution and the convergence order of approximation in weighted Sobolev space.
Recently Wang et.al. \cite{wang2021} considered the spectral Galerkin approximation for the optimal control problem governed by fractional diffusion-advection-reaction equations with \emph{integral} fractional Laplacian ( i.e.
\eqref{FOCPa1}-\eqref{admissibleset} for $\theta=1/2$). There are also some works on numerical methods for optimal control problems with a general two-sided space FDE constraints in standard Sobolev space. For example, a fast Gradient project method was exploited in \cite{du2016} to solve the discrete systems derived from two finite difference methods; a finite element method was considered in \cite{zhou2019}; a fast stochastic Galerkin method was developed in  \cite{du2018} when the constraint is  a random two-sided space-fractional diffusion equation.  

The aim of this work is to present an efficient spectral Petrov-Galerkin method based on the regularity analysis in weighted Soblev space for the optimal control problem \eqref{FOCPa1}-\eqref{admissibleset} with $\theta\in[0,1]$. To this end, we perform the following three steps.
\begin{enumerate}
	\item[-] Firstly, we construct the regularity of  \eqref{FOCPa1}-\eqref{admissibleset} in weighted Sobolev space.
	
	Due to the asymmetry of the fractional
operator $\mathcal{L}_\theta^\alpha$ for $\theta\neq1/2$, the regularity of state $u$ and adjoint state $z$ are analyzed in two different weighted Sobolev spaces  based on the coupled optimality system
\eqref{FOPCa3}-\eqref{FOPCc}.
To overcome this gap, the regularity connection between $(1-x)^{-\sigma}x^{-\sigma^*}u$ and $u$ is established in the weighted Sobolev space $H^{r}_{\omega^{\sigma , \sigma^*}} (\Omega)$, see Lemma \ref{solution_space},
where $\sigma$ and  $\sigma^*$ are constants defined by \eqref{sigma} depending on $\alpha$ and $\theta$.  The regularity results of  the state variable $u$, the adjoint state variable $z$ and the control variable $q$ are given in Theorem  \ref{regularity}.
\item[-]Secondly, we present a spectral Petrov-Galerkin method to \eqref{FOCPa1}-\eqref{admissibleset} and give its error estimation.

Based on the regularity results obtained in the first step,  we adopt weighted Jacobi polynomials as basis in the spectral Petrov-Galerkin method to recover accuracy of numerical solution from
the boundary singularity, and estimate the errors of approximations to  $u$ in $L^2_{\omega^{-\sigma,-\sigma^*}}$-norm, $z$ in $L^2_{\omega^{-\sigma^*,-\sigma}}$-norm and  $q$ in $L^2$-norm, see Theorem \ref{converth}.
\item[-]Finally, we provide a fast projected gradient algorithm with a quasi-linear complexity  to solve the resulting discrete system.

Compared to  the projected gradient method for solving the linear optimization system obtained from the spectral Petrov-Galerkin method,  which requires $O(N^2)$ storage and $O(N^3)$
computational complexity, see \cite{hinze2009, li2019, wang2021}, we follow the idea in \cite{hao2021} and present a fast projected gradient method with linear storage $O(N)$  and quasilinear computational cost $O(N\log^2N)$, see Algorithm \ref{alg:SA}.
\end{enumerate}

The rest of the paper is organized as follows. In Section \ref{section2}, we introduce some necessary notations and properties of Jacobi polynomials and weighted Sobolev spaces. In Section \ref{section3}, we derive the
first-order optimality condition  and study the regularity of solutions for the optimal control problem \eqref{FOCPa1}-\eqref{admissibleset}. Based on the regularity results, a spectral Petrov-Galerkin method is presented
and its error estimate is given in Section \ref{section4}. In Section \ref{section5}, a fast iteration algorithm is presented  to solve the discrete systems. Numerical examples given in this section verify our theoretical findings and show that the proposed spectral Petrov-Galerkin method with the fast algorithm is efficient.

\section{Preliminary}\label{section2}
	In this section, we will give some basic definitions and properties of  Jacobi polynomials and weighted Sobolev spaces.
	
	\subsection{Jacobi polynomials}
	For $\gamma , \beta > -1$, $n \in \mathbb{N}$ and $t \in [-1 , 1]$, $P_n^{\gamma , \beta}(t)$ is the classical Jacobi polynomial of degree $n$. Let $t = 2x-1$, then the domain of the Jacobi polynomials $P_n^{\gamma ,
\beta}(t)$ is transformed to $[0,1]$. Following this idea, we introduce
	\begin{equation*}
		Q_n^{\gamma , \beta}(x) = P_n^{\gamma , \beta}(2x-1) , \ x \in [0,1].
	\end{equation*}
	
	\begin{itemize}[leftmargin=*]
		\item \textbf{Orthogonality.} The Jacobi polynomials $Q_n^{\gamma , \beta}(x)$ are mutually orthogonal with respect to the weight $(1-x)^\gamma x^\beta$: for $\gamma, \beta > -1$,
		\begin{equation}\label{Orthogonality}
			 \int_{0}^{1} (1-x)^\gamma x^\beta Q_n^{\gamma,\beta}(x) Q_m^{\gamma,\beta}(x) dx = \delta_{mn} \| Q_n^{\gamma,\beta} \|^2_{\omega^{\gamma,\beta}},
		\end{equation}
		where $\delta_{mn}$ is the Kronecker symbol and
		\begin{equation*}
			\| Q_n^{\gamma,\beta} \|^2_{\omega^{\gamma,\beta}} = \frac{1}{2n+\gamma+\beta+1} \cdot \frac{\Gamma(n+\beta+1) \Gamma(n+\gamma+1)}{\Gamma(n+1) \Gamma(n+\gamma+\beta+1)}:=h_n^{\gamma,\beta}.
		\end{equation*}
	
		\item \textbf{Properties of Jacobi polynomials.}
		\begin{Lemma}[\cite{ervin2018}]\label{lem21}
			For the $n$-th order Jacobi polynomials $Q_n^{\sigma,\sigma^*}(x)$ and $Q_n^{\sigma^*,\sigma}(x)$, $x \in [0,1]$, it holds that
			\begin{equation}\label{ptzz}
				\mathcal{L}_\theta^\alpha \left[ (1-x)^\sigma x^{\sigma^*} Q_n^{\sigma,\sigma^*} (x) \right] = \lambda_{\theta,n}^\alpha Q_n^{\sigma^*,\sigma}(x) ,
			\end{equation}
			in which
			\begin{equation*}
				\lambda_{\theta,n}^\alpha = -\frac{\sin (\pi  \alpha)}{\sin (\pi \sigma^*) + \sin (\pi \sigma)} \frac{\Gamma(n+1+\alpha)}{\Gamma(n+1)},
			\end{equation*}
			$\sigma^* = \alpha - \sigma$ and $\sigma$ is determined by
			\begin{equation}\label{sigma}
				\theta = \frac{\sin (\pi (\alpha- \sigma))}{\sin (\pi (\alpha - \sigma)) + \sin (\pi \sigma)}.
			\end{equation}
		\end{Lemma}
		
		\begin{Remark}\label{remark2.2}	
			To ensure \eqref{sigma} uniquely solvable, we constrain $\sigma, \sigma^* \in (0,1]$. Note that in Lemma \ref{lem21}, $\sigma = \sigma^* = \alpha/2$ for $\theta = 1/2$ and $\sigma = 1, \sigma^* = \alpha-1$ for $\theta = 1$.
		\end{Remark}
		
		\begin{Lemma}[\cite{ervin2020}]
			For  Jacobi polynomials $Q_n^{\gamma , \beta}(x)$ with $\gamma , \beta > -1$, it holds that
			\begin{equation*}
				D^k \left[ Q_n^{\gamma , \beta}(x) \right] = \frac{\Gamma(n+k+\gamma+\beta+1)}{n+\gamma+\beta+1} Q_{n-k}^{\gamma+k , \beta+k}(x)
			\end{equation*}
			and
			\begin{equation}\label{ind}
				D^k \left[ (1-x)^{\gamma+k} x^{\beta+k} Q_{n-k}^{\gamma+k , \beta+k}(x) \right] = \frac{(-1)^k n!}{(n-k)!} (1-x)^\gamma x^\beta Q_{n}^{\gamma , \beta}(x) ,
			\end{equation}
			in which $D^k$ denotes $\frac{d^k}{dx^k}$ and $k = 1 , 2 , \cdots$.
		\end{Lemma}
		
	\end{itemize}
	
	\subsection{Weighted Sobolev spaces}
	\begin{itemize}[leftmargin=*]
		\item $\bm{L^2_{\omega^{\gamma,\beta}}(\Omega)}$. Denote $\omega^{\gamma,\beta}(x) = (1-x)^\gamma x^\beta, \ \gamma, \beta > -1$. Then
		\begin{equation*}
			L^2_{\omega^{\gamma,\beta}}(\Omega) = \{ v: \int_{\Omega} \omega^{\gamma,\beta}(x) v^2(x) dx < \infty  \}
		\end{equation*}
		equipped with the following inner product and norm
		\begin{equation*}
			(u,v)_{\omega^{\gamma,\beta}} = \int_{\Omega}\omega^{\gamma,\beta} uv  dx \ , \quad \| u \|_{\omega^{\gamma,\beta}} = \sqrt{ (u,u)_{\omega^{\gamma,\beta}} }.
		\end{equation*}
		
		\item \bm{$H^s_{\omega^{\gamma,\beta}} (\Omega)$}. The weighted Sobolev space with non-negative integer $s$ is defined as (see \cite{Babuska2002,guo1})
		\begin{equation*}
			H^s_{\omega^{\gamma,\beta}} (\Omega) = \left\{ v : D^k v(x) \in L^2_{\omega^{\gamma+k,\beta+k}} (\Omega) , k = 0,1,\cdots ,s \right\} ,
		\end{equation*}
		equipped with the norm
		\begin{equation*}
			\| v \|_{H^s_{\omega^{\gamma,\beta}}} = ( \sum_{k=0}^{s} |v|^2_{H^k_{\omega^{\gamma,\beta}} } )^{1/2} \ , \ \ |v|_{H^k_{\omega^{\gamma,\beta}} } = \| D^kv
\|_{\omega^{\gamma+k,\beta+k}} .
		\end{equation*}
		For $s \in \mathbb{R}^+$, $H^s_{\omega^{\gamma,\beta}} (\Omega)$ can be defined by interpolation via the K-method \cite{adams1975}. For $s<0$, it is defined by the (weighted) $L^2$ duality.
		
	
		\item {\bm{$H_{(\eta)}^s(J)$}. }  Let $\mathbb{N}_0 = \mathbb{N} \cup \{ 0 \}$ and $J=(0,\frac{3}{4})$, for $s\geq0$, $s=\lfloor s\rfloor+\nu$, $0\leq \nu <1$, where $\lfloor s\rfloor$ is the integer part of $s$,
		\begin{equation*}
			H_{(\eta)}^s(J)=\{v : v(x)\ \mbox{is measurable and}\ \|v\|_{H_{(    \eta    )}^s(J)}<\infty\},
		\end{equation*}
		in which the norm $\|\cdot\|_{H_{(\eta)}^s(J)}$ is defined by
		\begin{align*}
			\|v\|^2_{H^s_{(\eta)}(J)} =
			\begin{cases}
				\sum\limits_{k=0}^{\lfloor s \rfloor} \| D^kv \|^2_{L^2_{(\eta+k)}(J)}, &s \in \mathbb{N}_0 ,\\[5mm]
				\sum\limits_{k=0}^{\lfloor s \rfloor} \|D^kv\|^2_{L^2_{(\eta+k)} (J)} + |v|^2_{H^s_{(\eta)}(J)}, &s\in \mathbb{R}^{+}\backslash \mathbb{N}_0,
			\end{cases}
		\end{align*}
		and
		\begin{equation*}
			\|v\|_{L^2_{(\eta)}(J) }^2 = \int_J x^\eta v^2(x) dx\ ,\quad
			|v|_{H^s_{(\eta)}(J)}^2 = \iint_{\Lambda^*}x^{\eta+s} \frac{ \left| D^{\lfloor s \rfloor} v(x) - D^{\lfloor s\rfloor} v(y) \right|^2 }{ \left| x-y \right|^{1+2\nu}} dxdy
		\end{equation*}
		where
		\begin{equation*}
				\Lambda^* = \big\{  (x,y): \frac{2}{3}x < y < \frac{3}{2}x , \  0 < x < \frac{1}{2}  \big\}   \ \cup \ \big\{ (x,y) : \frac{3}{2}x - \frac{1}{2} < y < \frac{2}{3}x + \frac{1}{3} , \  \frac{1}{2} \leq
x<\frac{3}{4} \big\}.
		\end{equation*}

	\end{itemize}
		By the definition of $H^s_{\omega^{\gamma,\beta}}(\Omega)$ and $ H^s_{(\eta)}(J)$, the following result can be readily obtained.
		\begin{Lemma}[\cite{ervin2020,li2020}]\label{lemma2.3}
			A function $v \in H^s_{\omega^{\gamma,\beta}} (\Omega)$ if and only if $v \in H^s_{(\beta)} (J)$ and $\hat{v} \in H^s_{(\gamma)} (J)$, where $\hat{v}(x) := v(1-x)$.
		\end{Lemma}


\section{The optimal control problem}\label{section3}
	In this section, we derive the first-order optimality condition and analyze the regularity of the optimal control problem (\ref{FOCPa1})-(\ref{admissibleset}) in weighted Sobolev space.
We define the bilinear form
$A:H_0^{\frac{\alpha}{2}}(\Omega)\times H_0^{\frac{\alpha}{2}}(\Omega)\rightarrow \mathbb{R}$ as
\begin{equation*}
A(u,v):=\theta({}_{0}D_x^{\frac{\alpha}{2}}u,{}_{x}D_1^{\frac{\alpha}{2}}v)+(1-\theta)({}_{x}D_1^{\frac{\alpha}{2}}u,{}_{0}D_x^{\frac{\alpha}{2}}v)+\lambda_1( Du,v)+\lambda_2(u,v).
\end{equation*}
 Then the weak formulation of state equation \eqref{FOCPa1} is defined as
\begin{equation*}
A(u,v)=( f+q,v), \ \forall v\in H_0^{\frac{\alpha}{2}}(\Omega),
\end{equation*}
from \cite{ervin2006}, which admits a unique solution $u\in H_0^{\frac{\alpha}{2}}(\Omega)$.
Therefore, we can formulate the optimal control problem (\ref{FOCPa1})-(\ref{admissibleset}) as:
\begin{equation}\label{FOCPa4}
		  \min\limits_{q\in U_{ad}}J(u,q) = \frac{1}{2}\|u-u_d\|^2_{L^2(\Omega)} + \frac{\gamma}{2}\|q\|^2_{L^2(\Omega)}
	\end{equation}
	subject to
	\begin{eqnarray}\label{FOCPb4}
	A(u,v)=( f+q,v), \ \forall v\in H_0^{\frac{\alpha}{2}}(\Omega).		
	\end{eqnarray}

 Let $\hat{J}(q):=J(u(q),q)$. With this notation, the optimal control problem (\ref{FOCPa4})-(\ref{FOCPb4}) can be written as a reduced optimization problem:
\begin{equation*}
\min\limits_{q\in U_{ad}}\hat{J}(q).
\end{equation*}
Note that the admissible set $U_{ad}$ is closed and convex, and cost functional $\hat{J}$ is strictly convex, then it admits a unique solution by a routine argument \cite{lions,hinze}. The existence and uniqueness of the solution to the optimal control problem  follows this result.

\subsection{The first-order optimality condition}
 \begin{Lemma}[\cite{ervin2006}]\label{change_order}
		 Suppose that $u$, $v\in L^2(\Omega)$ and vanish on boundary $\partial\Omega$, then it holds that
		\begin{equation*}
			( \mathcal{L}_\theta^\alpha u , v ) = ( u ,  \mathcal{L}_{1-\theta}^\alpha v ).
		\end{equation*}
	\end{Lemma}
	\begin{Theorem}\label{R1}
		Suppose that $q\in U_{ad}$ is an optimal control for the problem (\ref{FOCPa1})-(\ref{admissibleset}) and $u$ is the associated state variable. Then there exists an adjoint state variable $z$, such that $(u,z,q)$ satisfies the following optimality conditions
			\begin{align}\label{FOPCa3}
				\begin{cases}
					\mathcal{L}_\theta^\alpha u + \lambda_1 Du + 	\lambda_2 u = f(x)+q(x), &x\in \Omega, \\[2mm]
					u(x) = 0, &x \in \partial \Omega ,
				\end{cases}
			\end{align}
			\begin{align}\label{FOPCb3}
				\begin{cases}
					\mathcal{L}_{1-\theta}^\alpha z - \lambda_1 Dz + \lambda_2 z =u(x)-u_d(x), &x\in \Omega, \\[2mm]
					z(x) = 0, &x \in \partial \Omega ,
				\end{cases}
			\end{align}
			and the variational inequality
			\begin{equation}\label{FOPCc}
				\int_\Omega (\gamma q+z)(v-q) dx \geq 0, \ \  	v\in U_{ad}.
			\end{equation}
	\end{Theorem}

\begin{proof}

We denote by $f'(w)(h)$  the Frech\'et derivative of $f$ at $w$ in the direction $h$. The first-order necessary (and, owing to the convexity of cost functional $\hat{J}$, also sufficient) optimality condition takes the form
\begin{equation}\label{oc}
\hat{J}'(q)(v-q)=\int_{\Omega}u'(q)(v-q)(u-u_d)dx+\int_{\Omega}\gamma q(v-q)dx\geq0,\ \forall v\in U_{ad},
\end{equation}
 In view of the state equation, $\hat{u}:=u'(q)(v-q)$ satisfies
\begin{eqnarray*}\left\{\begin{aligned}
					&\mathcal{L}_\theta^\alpha\hat{u} + \lambda_1 D\hat{u} + 	\lambda_2 \hat{u} = v(x)-q(x), \ &x\in \Omega, \\
					&\hat{u}(x) = 0, \ &x \in \partial \Omega.				
\end{aligned}\right.\end{eqnarray*}
Introducing the following adjoint state equation
\begin{eqnarray*}\left\{\begin{aligned}
					&\mathcal{L}_{1-\theta}^\alpha z - \lambda_1 Dz + \lambda_2 z =u(x)-u_d(x), \ &x\in \Omega, \\
					&z(x) = 0,\ &x \in \partial \Omega.
\end{aligned}\right.\end{eqnarray*}
Then by using Lemma \ref{change_order}, we have
\begin{align*}
\int_{\Omega}\hat{u}(u-u_d)dx&=\int_{\Omega}\hat{u}(\mathcal{L}_{1-\theta}^\alpha z - \lambda_1 Dz + \lambda_2 z)dx\\
&=\int_{\Omega}(\mathcal{L}_{\theta}^\alpha\hat{u}+\lambda_1 D\hat{u}+\lambda_2 \hat{u})zdx=\int_{\Omega}(v-q)zdx.
\end{align*}
With this, \eqref{oc} becomes
$$\hat{J}'(q)(v-q)=\int_\Omega (\gamma q+z)(v-q) dx \geq 0, \   	v\in U_{ad}.$$
	\end{proof}
	\begin{Remark}
		According to \cite{chen2008}, the variational inequality \eqref{FOPCc} is equivalent to the following condition
	\begin{equation}\label{remark3.3}
		\gamma q = \max\{0,\bar{z}\} - z,
	\end{equation}
	in which $\bar{z} = \frac{1}{|\Omega|} \int_{\Omega} z(x) dx$ and $|\Omega|$ is the interval size.
	\end{Remark}

\subsection{Regularity for optimal control problem in weighted Sobolev space}
	
	\begin{Lemma}[\cite{ervin2020,li2020}]\label{lemma3.2}
		Suppose $n \leq s < n+1$, $n \in \mathbb{N}_0$, $p \geq 0$, $\mu > -1$ and $\psi \in H^s_{(\mu)} (J)$. If
		\begin{equation*}
			0 \leq \vartheta \leq s \ , \ \ \varsigma + 2p \geq \mu \ , \ \ \varsigma + 2p - \vartheta > -1 \ , \ \ \varsigma + 2p + \vartheta \geq \mu + s,
		\end{equation*}
		then $t^p\psi \in H^\vartheta_{(\varsigma)} (J)$. Moreover, there exists a positive constant $C$ independent of $\psi$, such that
		\begin{equation*}
			\| t^p \psi \|_{H^\vartheta_{(\varsigma)} (J)} \leq C \| \psi \|_{H^s_{(\mu)} (J)}.
		\end{equation*}
	\end{Lemma}
	
	\begin{Lemma}[\cite{ervin2020}]\label{lemma3.3}
		Let $s \geq 0$, $\phi \in H^s_{(\mu)} (J)$, $\mu>-1$, and $g \in C^{\lceil s\rceil}(J)$, then
		\begin{equation*}
			\| g \phi \|_{H^s_{(\mu)} (J)} \leq C \| g \|_{C^{[s]}(J)} \| \phi \|_{H^s_{(\mu)} (J)},
		\end{equation*}
		where $\lceil s\rceil$ is the smallest integer greater than $s$.
	\end{Lemma}

	\begin{Lemma}\label{solution_space}
		If $\omega^{-\sigma , -\sigma^*} u \in H^s_{\omega^{\sigma , \sigma^*}} (\Omega)$ and $s < \min \{ 5\sigma + 1 , 5\sigma^* + 1 \}$, then $u \in H^\vartheta_{\omega^{\sigma , \sigma^*}} (\Omega)$ for $
\vartheta = \min \{ s , 3\sigma + 1 - \epsilon , 3\sigma^* + 1 - \epsilon \}$ with arbitrarily small $\epsilon > 0$.
	\end{Lemma}	

	\begin{proof}
		According to Lemma \ref{lemma2.3},
		\begin{align*}
			\omega^{-\sigma , -\sigma^*} u \in H^s_{\omega^{\sigma , \sigma^*}} (\Omega) \ \ \Leftrightarrow \ \ &\omega^{-\sigma , -\sigma^*} u = (1-x)^{-\sigma} x^{-\sigma^*} u \in H^s_{(\sigma^*)} (J) \ \ \text{and}\\[2mm]
			&\hat{\omega}^{-\sigma , -\sigma^*} \hat{u} = x^{-\sigma} (1-x)^{-\sigma^*} \hat{u} \in H^s_{(\sigma)} (J) .
		\end{align*}
		\begin{itemize}[leftmargin=*]
			\item On one hand, from $(1-x)^{-\sigma} x^{-\sigma^*} u \in H^s_{(\sigma^*)} (J)$ and Lemma \ref{lemma3.3} $\big( \text{obviously} \ (1-x)^\sigma \in C^\infty (J) \big) $ we know that
			\begin{equation*}
				x^{-\sigma^*} u = (1-x)^{\sigma} \cdot 	(1-x)^{-\sigma} x^{-\sigma^*} u \in H^s_{(\sigma^*)} (J).
			\end{equation*}
			According to Lemma \ref{lemma3.2}, let $s = s$, $\mu =  p =  \varsigma = \sigma^*$, we have $\vartheta = \min \{ s , 3\sigma^* + 1 - \epsilon \}$.  In fact,
			\begin{align*}
				\begin{cases}
					0 \leq \vartheta \leq s , \\[1mm]
					\vartheta < 3\sigma^* + 1 , \\[1mm]
					\vartheta + 2\sigma^* \geq s ,
				\end{cases}
			\end{align*}
			which results in
			\begin{equation}\label{u1}
				u \in H^{\min \{ s ,\, 3\sigma^* + 1 - \epsilon \}}_{(\sigma^*)} (J).
			\end{equation}
			
			\item On the other hand, from $x^{-\sigma} (1-x)^{-\sigma^*} \hat{u} \in H^s_{(\sigma)} (J)$ and Lemma \ref{lemma3.3} we know that
			\begin{equation*}
				x^{-\sigma} \hat{u} = (1-x)^{\sigma^*} \cdot 	(1-x)^{-\sigma^*} x^{-\sigma} \hat{u} \in H^s_{(\sigma)} (J).
			\end{equation*}
			According to Lemma \ref{lemma3.2}, let $s = s$, $\mu = p = \varsigma = \sigma$, we have $\vartheta = \min \{ s ,\, 3\sigma + 1 - \epsilon \}$. Indeed,			
			\begin{align*}
				\begin{cases}
					0 \leq \vartheta \leq s,  \\[1mm]
					\vartheta < 3\sigma + 1 , \\[1mm]
					\vartheta + 2\sigma \geq s ,
				\end{cases}
			\end{align*}
			which implies
			\begin{equation}\label{u2}
				\hat{u} \in H^{\min \{ s ,\, 3\sigma + 1 - \epsilon \}}_{(\sigma)} (J) .
			\end{equation}
		\end{itemize}
		Putting \eqref{u1} and \eqref{u2} together and according to Lemma \ref{lemma2.3}, it can be deduced that
		\begin{equation*}
			u \in H^{\min \{ s ,\, 3\sigma + 1 - \epsilon ,\, 3\sigma^* + 1 - \epsilon \}}_{\omega^{\sigma , \sigma^*}} (\Omega).
		\end{equation*}
	\end{proof}

	\begin{Lemma}\label{regulity_state}
		For the state equation \eqref{FOPCa3} with right hand term $\tilde{f} \in H^r_{\omega^{\sigma^* , \sigma}} (\Omega)$, $r \geq 0$, there exists a unique solution $u$ such that
		\begin{equation*}\label{regulityu}
			\omega^{-\sigma , -\sigma^*}u \in H^{\min \{ r + \alpha ,\, s \}}_{\omega^{\sigma , \sigma^*}} (\Omega),\, u \in H^{\min \{ r + \alpha ,\, s \}}_{\omega^{\sigma , \sigma^*}} (\Omega),
		\end{equation*}
		in which $s =2\alpha + \min\{\sigma,\sigma^*\} - 1 - \epsilon $.
	\end{Lemma}
\begin{proof}
		The well-posedness and regularity of the state equation \eqref{FOPCa3} has been already established in \cite{ervin2020}: For the state equation \eqref{FOPCa3} with right hand term $\tilde{f} \in H^r_{\omega^{\sigma^* , \sigma}}
(\Omega)$, we have
		\begin{equation*}
			\omega^{-\sigma , -\sigma^*} u \in H^{\min \{ r + \alpha ,\, s \}}_{\omega^{\sigma , \sigma^*}} (\Omega).
		\end{equation*} 	
		It is easy to verify that
		\begin{equation*}
			\min \{ r + \alpha , s \} < \min \{ 5\sigma + 1 , 5\sigma^* + 1 \},
		\end{equation*}
		so all the assumptions in Lemma \ref{solution_space} are satisfied, which leads to
		\begin{equation*}
			u \in H^{\min \{ r + \alpha , \,s  ,\, 3\sigma + 1 - \epsilon ,\, 3\sigma^* + 1 - \epsilon \} } _{\omega^{\sigma , \sigma^*}} (\Omega).
		\end{equation*}
		According to $\sigma^*+\sigma=\alpha$ and $\sigma^*,\,\sigma\in(0,1]$ in Remark \ref{remark2.2}, we know that $3\min\{\sigma ,\, \sigma^*\}+ 1 - \epsilon> s$, which implies
		\begin{equation*}
			u \in H^{\min \{ r + \alpha ,\, s \}}_{\omega^{\sigma , \sigma^*}} (\Omega) .
		\end{equation*}
	\end{proof}
A similar argument can be applied to the adjoint state equation to derive the following result.
	\begin{Lemma}\label{regularityz}
	 For the adjoint state equation \eqref{FOPCb3} with right hand term $g \in H^r_{\omega^{\sigma , \sigma^*}} (\Omega)$, $r \geq 0$, there exists a unique solution $z$ such
that
	\begin{equation*}
		\omega^{-\sigma^* , -\sigma}z\in H^{\min \{ r + \alpha , s \}}_{\omega^{\sigma^* , \sigma}} (\Omega),\ z \in H^{\min \{ r + \alpha , s \}}_{\omega^{\sigma^* , \sigma}} (\Omega)
	\end{equation*}
	in which $s =2\alpha + \min\{\sigma,\sigma^*\} - 1 - \epsilon $.
  \end{Lemma}

%
%
%
	
	\begin{Theorem}\label{regularity}
		Assume that $(u,z,q)$ is the solution of the optimality system \eqref{FOCPa1}-\eqref{admissibleset}, if $f \in H^{r}_{\omega^{\sigma^* , \sigma}}(\Omega)$, $u_d \in H^{r}_{\omega^{\sigma , \sigma^*}}(\Omega) $ and $q \in
L^2(\Omega)$, $r \geq 0$, then the regularity of the  state  $u$, adjoint state  $z$ and control  $q$ satisfy		
			$$ u \in H^{\min \{ r + \alpha , s \}}_{\omega^{\sigma , \sigma^*}} (\Omega),\, z\in H^{\min \{ r + \alpha , s \}}_{\omega^{\sigma^* , \sigma}} (\Omega),\ and \ q \in H^{\min \{ r + \alpha , s \}}_{\omega^{\sigma^* , \sigma}} (\Omega),$$		
respectively. Moreover, we have
\begin{equation*}
			\omega^{-\sigma , -\sigma^*} u \in H^{\min \{ r + \alpha , s \}}_{\omega^{\sigma , \sigma^*}} (\Omega),\;\; \omega^{-\sigma^* , -\sigma}z  \in H^{\min \{ r + \alpha , s \}}_{\omega^{\sigma^* , \sigma}} (\Omega),\ \text{and}\ \omega^{-\sigma^* , -\sigma}q  \in H^{\min \{ r + \alpha , s \}}_{\omega^{\sigma^* , \sigma}} (\Omega)
		\end{equation*}
		 where $s =  2\alpha + \min(\sigma,\sigma^*) - 1 - \epsilon   $.
	\end{Theorem}
	
	\begin{proof}
		Note that $f \in H^{r}_{\omega^{\sigma^* , \sigma}}(\Omega)$ and $q \in L^2(\Omega) \subset H^{0}_{\omega^{\sigma^* , \sigma}}(\Omega)$, thus $f + q \in H^{0}_{\omega^{\sigma^* , \sigma}}(\Omega)$  in equation \eqref{FOPCa3}. According to $\alpha\in(1,2)$ and  $\sigma^*,\,\sigma\in(0,1]$ we observe
		\begin{equation*}
			\alpha<s =  2\alpha + \min(\sigma,\sigma^*) - 1 - \epsilon  <2 \alpha,
		\end{equation*}
		then by Lemma \ref{regulity_state} it follows that
		\begin{equation*}
			 u \in H^{\min \left\{ \alpha , s \right\} }_{\omega^{\sigma , \sigma^*}}(\Omega) =  H^{\alpha}_{\omega^{\sigma , \sigma^*}}(\Omega).
		\end{equation*}
		Besides, for $u_d \in  H^{r}_{\omega^{\sigma , \sigma^*}}(\Omega)$, $u - u_d \in  H^{\min \{ \alpha , r \}}_{\omega^{\sigma , \sigma^*}}(\Omega)$ in equation \eqref{FOPCb3}.  By using Lemma \ref{regularityz} and noting that $s<2\alpha$, we have
		\begin{equation*}
			\omega^{-\sigma^*,-\sigma}z\in H^{\min \{ r+\alpha , 2\alpha , s \}}_{\omega^{\sigma^* , \sigma}}(\Omega)  =  H^{\min \{ r+\alpha , s \}}_{\omega^{\sigma^* , \sigma}}(\Omega), \quad z \in  H^{\min \{ r+\alpha , s \}}_{\omega^{\sigma^* , \sigma}}(\Omega).
		\end{equation*}
		  From \eqref{remark3.3}, $q$ has the same regularity as $z$, i.e.
		 \begin{equation*}
		 \omega^{-\sigma^*,-\sigma}q\in	H^{\min \{ r+\alpha , s \}}_{\omega^{\sigma^* , \sigma}}(\Omega),\ q \in H^{\min \{ r+\alpha , s \}}_{\omega^{\sigma^* , \sigma}}(\Omega).
		 \end{equation*}
		 Then for $f + q \in H^{\min \{ r , s \}}_{\omega^{\sigma^* , \sigma}}(\Omega)$ in \eqref{FOPCa3}, by Lemma \ref{regulity_state}		
		 we can obtain that
		 \begin{equation*}
		 	\omega^{-\sigma,-\sigma^*}u\in H^{\min \{ r+\alpha , s \}}_{\omega^{\sigma , \sigma^*}}(\Omega),\quad u \in H^{\min \{ r+\alpha , s \}}_{\omega^{\sigma , \sigma^*}}(\Omega).
		 \end{equation*}
		  This completes the proof.
	\end{proof}
\begin{Remark}
Note that  the regularity of solution $(u,z,q)$ given by Theorem \ref{regularity} can not be further improved.
\end{Remark}
\section{Spectral Petrov-Galerkin method}\label{section4}
In this section, we consider a spectral Petrov-Galerkin method for the space fractional optimal control problem \eqref{FOCPa1}-\eqref{admissibleset} and present its  error estimates.

Let $a:L^2_{\omega^{-\sigma,-\sigma^*}}(\Omega)\times H^{\alpha}_{\omega^{\sigma^*,\sigma}}(\Omega)\rightarrow \mathbb{R}$, $b:L^2_{\omega^{-\sigma^*,-\sigma}}(\Omega)\times H^{\alpha}_{\omega^{\sigma,\sigma^*}}(\Omega)\rightarrow \mathbb{R}$,
\begin{eqnarray*}\begin{aligned}
a(u,v)&:=(u,\mathcal{L}_{1-\theta}^{\alpha}(\omega^{\sigma^*,\sigma}v))-\lambda_1(u,D(\omega^{\sigma^*,\sigma}v))+\lambda_2(u,v)_{\omega^{\sigma^*,\sigma}},\\
b(z,w)&:=(z,\mathcal{L}_{\theta}^{\alpha}(\omega^{\sigma,\sigma^*}w))+\lambda_1(z,D(\omega^{\sigma,\sigma^*}w))+\lambda_2(z,w)_{\omega^{\sigma,\sigma^*}}.
\end{aligned}\end{eqnarray*}
We consider the spectral Petrov-Galerkin ultra-weak formulation for the optimal control problem \eqref{FOCPa1}-\eqref{admissibleset}: Given $f\in H^{r}_{\omega^{\sigma^*,\sigma}}(\Omega)$,  $u_d\in H^{r}_{\omega^{\sigma,\sigma^*}}(\Omega)$
with $r\geq0$, to find $(u,q)\in L^2_{\omega^{-\sigma,-\sigma^*}}(\Omega)\times U_{ad}$ such that
\begin{eqnarray*}
\min\limits_{q\in U_{ad}}J(u,q)=\frac{1}{2}\|u-u_d\|^2+\frac{\gamma}{2}\|q\|^2
\end{eqnarray*}
subject to
\begin{align*}
a(u,v)=(f+q,v)_{\omega^{\sigma^*,\sigma}},\ \forall v\in H^{\alpha}_{\omega^{\sigma^*,\sigma}}(\Omega).
\end{align*}
 Then the spectral Petrov-Galerkin ultra-weak formulation of the first-order optimality condition  is: Given $f\in H^{r}_{\omega^{\sigma^*,\sigma}}(\Omega)$,  $u_d\in H^{r}_{\omega^{\sigma,\sigma^*}}(\Omega)$ with $r\geq0$, to find
$(u,z,q)\in L^2_{\omega^{-\sigma,-\sigma^*}}(\Omega)\times L^2_{\omega^{-\sigma^*,-\sigma}}(\Omega) \times U_{ad}$ such that
\begin{subequations}\label{uweak}
\begin{numcases}
{}a(u,v)=(f+q,v)_{\omega^{\sigma^*,\sigma}},\ \forall v\in H^{\alpha}_{\omega^{\sigma^*,\sigma}}(\Omega),\label{uweaka}\\ 
b(z,w)=(u-u_d,w)_{\omega^{\sigma,\sigma^*}},\ \forall w\in H^{\alpha}_{\omega^{\sigma,\sigma^*}}(\Omega),\label{uweakb}\\
(\gamma q+z,v-q)\geq0,\ \forall v \in U_{ad}\label{uweakc}.
\end{numcases}
\end{subequations}

We introduce the finite dimensional spaces
\begin{eqnarray*}\begin{aligned}
U_N&=\{u|u=\omega^{\sigma,\sigma^*} v, v\in W_N\},\  W_N=\textrm{span}\{Q_m^{\sigma,\sigma^*}\}_{m=0}^N \subset H^{\alpha}_{\omega^{\sigma,\sigma^*}}(\Omega),\\
Z_N&=\{z|z=\omega^{\sigma^*,\sigma} v, v\in V_N\},\ V_N=\textrm{span}\{Q_m^{\sigma^*,\sigma}\}_{m=0}^N\subset H^{\alpha}_{\omega^{\sigma^*,\sigma}}(\Omega).
\end{aligned}\end{eqnarray*}
The spectral Petrov-Galerkin scheme for the optimal control problem is: Given $f\in H^{r}_{\omega^{\sigma^*,\sigma}}(\Omega)$,  $u_d\in H^{r}_{\omega^{\sigma,\sigma^*}}(\Omega)$ with $r\geq0$, to find $(u_N,q_N)\in U_N \times U_{ad}$ such
that
\begin{eqnarray*}\label{scheme1}
\min\limits_{q_N\in U_{ad}}J(u_N,q_N)=\frac{1}{2}\|u_N-u_d\|^2+\frac{\gamma}{2}\|q_N\|^2
\end{eqnarray*}
subject to
\begin{align*}
a(u_N,v_N)=(f+q_N,v_N)_{\omega^{\sigma^*,\sigma}},\ \forall v_N\in V_N.
\end{align*}
Similar to the infinite dimensional problem  (\ref{FOCPa1})-(\ref{admissibleset}), the above discrete  problem also admits a unique solution.

In order to obtain the discrete first-order optimality condition of the optimal control problem, we define a Lagrange functional as
$$L(u_N,z_N,q_N):=J(u_N,q_N)+(f+q_N,z_N)-(u_N,\mathcal{L}_{1-\theta}^{\alpha}z_N)+\lambda_1(u_N,Dz_N)-\lambda_2(u_N,z_N).$$
Then, the discrete optimality conditions can be deduced by computing the Fr\'echet derivatives as follows:
$$\frac{\p L(u_N,z_N,q_N)}{\p u_N}(w_N)=0, \ \forall w_N \in U_N\ \ \text{and} \ \
\frac{\p L(u_N,z_N,q_N)}{\p q_N}(v-q_N)\geq0,\ \forall v \in U_{ad},$$
and the following equations can be obtained
\begin{align}
(\mathcal{L}_{1-\theta}^{\alpha}z_N,w_N)-\lambda_1(Dz_N,w_N)+\lambda_2(z_N,w_N)&=(u_N-u_d,w_N),\ \forall w_N\in U_N, \nonumber\\
(\gamma q_N+z_N,v-q_N)&\geq0,\ \forall v \in U_{ad}.\nonumber
\end{align}
Therefore, the discrete first-order optimality condition can be summarized as: Given $f\in H^{r}_{\omega^{\sigma^*,\sigma}}(\Omega)$,  $u_d\in H^{r}_{\omega^{\sigma,\sigma^*}}(\Omega)$ with $r\geq0$, to find $(u_N,z_N,q_N)\in U_N\times
Z_N \times U_{ad}$ such that

\begin{subequations}\label{adjoint}
\begin{numcases}{}
{}a(u_N,v_N)=(f+q_N,v_N)_{\omega^{\sigma^*,\sigma}},\ \forall v_N\in V_N.\label{adjointa}\\
b(z_N,w_N)=(u_N-u_d,w_N)_{\omega^{\sigma,\sigma^*}},\ \forall w_N\in W_N.\label{adjointb}\\
(\gamma q_N+z_N,v-q_N)\geq0,\ \forall v \in U_{ad}.\label{adjointc}
\end{numcases}
\end{subequations}

To establish the error estimates of this problem, we need to consider the following auxiliary problem:
\begin{subequations}\label{au}
\begin{numcases}{}
a(u_N(q),v_N)=(f+q,v_N)_{\omega^{\sigma^*,\sigma}},\ \forall v_N\in V_N.\label{aua}\\
b(z_N(q),w_N)=(u_N(q)-u_d,w_N)_{\omega^{\sigma,\sigma^*}},\ \forall w_N\in W_N.\label{aub}\\
b(z_N(u),w_N)=(u-u_d,w_N)_{\omega^{\sigma,\sigma^*}},\ \forall w_N\in W_N.\label{auc}
\end{numcases}
\end{subequations}

\begin{Lemma}[\cite{hao2021,zheng2020}] There exists a constant $C>0$, such that for N sufficiently large,
\begin{eqnarray}\label{infsup1}
 \sup\limits_{0\not=v_N\in V_N}\frac{a(u_N,v_N)}{\|v_N\|_{H^{\alpha}_{\omega^{\sigma^*,\sigma}}}} \geq C\left\|u_{N}\right\|_{{\omega^{-\sigma,-\sigma^*}}}, \quad \forall u_{N} \in U_{N}, u_N \neq 0.
\end{eqnarray}
\end{Lemma}
\begin{Lemma}[\cite{hao2021,zheng2020}]\label{u} Assume that $u$ and $u_N$ are solutions of the state equation \eqref{uweaka} and its discrete counterpart with right hand term $\tilde{f}$, respectively. If $\tilde{f}\in
H_{\omega^{\sigma^*,\sigma}}^r(\Omega)$, $r\geq0$, then there exists a number $N_0>0$, such that when $N>N_0$, we  have the following error estimate
\begin{align}
\|u-u_N\|_{\omega^{-\sigma,-\sigma^*}}\leq CN^{-m}|\omega^{-\sigma,-\sigma^*}u|_{H^m_{\omega^{\sigma,\sigma^*}}},\label{ues}
\end{align}
where $m=\min \{ r + \alpha , 2\alpha + \min(\sigma,\sigma^*) - 1 - \epsilon\}.$
\end{Lemma}
An analogous argument can be applied to the adjoint state equation to obtain the Lemmas as follows.
\begin{Lemma} There exists a constant $C>0$, such that for N sufficiently large,
\begin{eqnarray}\label{infsup2}
 \sup\limits_{0\not=w_N\in W_N}\frac{b(z_N,w_N)}{\|w_N\|_{H_{\omega^{\sigma,\sigma^*}}^{\alpha}}} \geq C\left\|z_{N}\right\|_{{\omega^{-\sigma^*,-\sigma}}}, \quad \forall  z_{N} \in Z_{N}, z_N\neq 0.
\end{eqnarray}
\end{Lemma}

\begin{Lemma}\label{z}Assume that $z$ and $z_N$ are solutions of the adjoint state equation \eqref{uweakb} and its discrete counterpart with right hand term $g$, respectively. If $g\in H_{\omega^{\sigma,\sigma^*}}^r(\Omega)$,
$r\geq0$, then there exists a number $N_0>0$, such that when $N>N_0$, we  have the following error estimates
\begin{align}
\|z-z_N\|_{\omega^{-\sigma^*,-\sigma}}\leq CN^{-m}|\omega^{-\sigma^*,-\sigma}z|_{H^m_{\omega^{\sigma^*,\sigma}}},\label{zes}
\end{align}
where $m=\min \{ r + \alpha , 2\alpha + \min(\sigma,\sigma^*) - 1 - \epsilon\}.$
\end{Lemma}


\begin{Theorem}\label{converth}
Assume that $(u,z,q)$ and $(u_N,z_N,q_N)$ are solutions of \eqref{uweak} and \eqref{adjoint}, respectively. If $f\in H^{r}_{\omega^{\sigma^*,\sigma}}(\Omega)$,  $u_d\in H^{r}_{\omega^{\sigma,\sigma^*}}(\Omega)$ with
$r\geq0$,  we have the following error estimate
\begin{equation}\label{conver}
\left\|u-u_{N}\right\|_{{\omega^{-\sigma,-\sigma^*}}} +\|z-z_{N}\|_{{\omega^{-\sigma^*,-\sigma}}}+\|q-q_N\| \leq C N^{-m},
\end{equation}
where $m=\min \{ r + \alpha , 2\alpha + \min(\sigma,\sigma^*) - 1 - \epsilon\}.$
\end{Theorem}
\begin{proof}
The errors can be expressed by
\begin{align*}
u-u_N&=u-u_N(q)+u_N(q)-u_N,\\
z-z_N&=z-z_N(u)+z_N(u)-z_N.
\end{align*}
From \eqref{uweak},  \eqref{aua} and \eqref{auc}, we know that $u_N(q)$ and $z_N(u)$ are the spectral Petrov-Galerkin approximation of $u$ and $z$, respectively. By the estimates \eqref{ues}, \eqref{zes} and Theorem
\ref{regularity}, we have
\begin{align}
\|u-u_N(q)\|_{\omega^{-\sigma,-\sigma^*}}&\leq CN^{-m},\label{e1}\\
\|z-z_N(u)\|_{\omega^{-\sigma^*,-\sigma}}&\leq CN^{-m},\label{e3}
\end{align}
where $m=\min \{ r + \alpha , 2\alpha + \min(\sigma , \sigma^*) - 1 - \epsilon \}$ determined by the regularity of state and adjoint state in Theorem \ref{regularity}.

Firstly, we give the error  estimate of $u_N$. Subtracting \eqref{aua} from \eqref{adjointa}, it follows
\begin{align}
a(u_N-u_N(q),v_N)&=(u_N-u_N(q),\mathcal{L}_{1-\theta}^{\alpha}(\omega^{\sigma^*,\sigma}v_N)-\lambda_1 D(\omega^{\sigma^*,\sigma}v_N)+\lambda_2\,\omega^{\sigma^*,\sigma}v_N)\nonumber\\
&=(q_N-q,v_N)_{\omega^{\sigma^*,\sigma}},\ \forall v_N\in V_N.\label{zhong1}
\end{align}
By \eqref{infsup1} and the Cauchy-Schwarz inequality we have
\begin{align}
C\|u_N-u_N(q)\|_{\omega^{-\sigma,-\sigma^*}}&\leq \sup\limits_{0\not=v_N\in V_N}\frac{a(u_N-u_N(q),v_N)}{\|v_N\|_{H^{\alpha}_{\omega^{\sigma^*,\sigma}}}}\nonumber\\
 &\leq \sup\limits_{0\not=v_N\in V_N}\frac{\|q-q_N\|_{\omega^{\sigma^*,\sigma}}\|v_N\|_{\omega^{\sigma^*,\sigma}}}{\|v_N\|_{H^{\alpha}_{\omega^{\sigma^*,\sigma}}}}\leq\|q-q_N\|\label{e2}
 \end{align}
 This combined with \eqref{e1} implies
 \begin{align}
 \|u-u_N\|_{\omega^{-\sigma,-\sigma^*}}\leq CN^{-m}+C\|q-q_N\|.\label{uerror}
 \end{align}

Secondly, The error estimate of $z_N$ can be obtained in a similar way. Substracting \eqref{auc} from \eqref{adjointb} gives
 \begin{align}
 b(z_N-z_N(u),w_N)
 =(u_N-u,w_N)_{\omega^{\sigma,\sigma^*}},\ \forall w_N\in W_N\label{zhong2}.
 \end{align}
Then by  \eqref{infsup2} and the Cauchy-Schwarz inequality we have
 \begin{align}
C\|z_N-z_N(u)\|_{\omega^{-\sigma^*,-\sigma}}&\leq \sup\limits_{0\not
=w_N\in W_N}\frac{b(z_N-z_N(u),w_N)}{\|w_N\|_{H_{\omega^{\sigma,\sigma^*}}^{\alpha}}}\nonumber\\
&\leq \sup\limits_{0\not=w_N\in W_N}\frac{\|u-u_N\|_{\omega^{\sigma,\sigma^*}}\|w_N\|_{\omega^{\sigma,\sigma^*}}}{\|w_N\|_{H_{\omega^{\sigma,\sigma^*}}^{\alpha}}}
\leq \|u-u_N\|_{\omega^{-\sigma,-\sigma^*}}.\label{zNu}
\end{align}
Combining with \eqref{e3} and \eqref{uerror}, we get
\begin{align}
\|z-z_N\|_{\omega^{-\sigma^*,-\sigma}}\leq CN^{-m}+C\|q-q_N\|.\label{zerror}
\end{align}
Finally, we determine the error estimate of $q_N$. By subtracting \eqref{adjointb} from \eqref{aub} we have
\begin{align}
b(z_N(q)-z_N,w_N)&=(z_N(q)-z_N,\mathcal{L}_{\theta}^{\alpha}(\omega^{\sigma,\sigma^*}w_N)+\lambda_1 D(\omega^{\sigma,\sigma^*}w_N)+\lambda_2\,\omega^{\sigma,\sigma^*}w_N)\nonumber\\&
=(u_N(q)-u_N,w_N)_{\omega^{\sigma,\sigma^*}},\ \forall w_N\in W_N.\label{zhong3}
\end{align}
It follows from \eqref{FOPCc} with $v=q_N$ and \eqref{adjointc} with $v=q$ that
\begin{align*}
\gamma\left\|q-q_N\right\|^2 &=\int_{\Omega} \gamma q\left(q-q_{N}\right)dx-\int_{\Omega} \gamma q_{N}\left(q-q_{N}\right)dx \\
& \leq \int_{\Omega} z\left(q_{N}-q\right)dx+\int_{\Omega} z_{N}\left(q-q_{N}\right)dx \\
&=\int_{\Omega}\left(z-z_N\left(q\right)\right)\left(q_{N}-q\right)dx+\int_{\Omega}\left(z_N\left(q\right)-z_{N}\right)\left(q_{N}-q\right)dx.
\end{align*}
 By taking $\omega^{\sigma^*,\sigma}v_N=z_N\left(q\right)-z_{N}$ in \eqref{zhong1} and $\omega^{\sigma,\sigma^*}w_N=u_N-u_N(q)$ in \eqref{zhong3}, we have
\begin{align*}
&\int_{\Omega}\left(z_N\left(q\right)-z_{N}\right)\left(q_{N}-q\right)dx\\
&=(u_N-u_N(q),\mathcal{L}_{1-\theta}^{\alpha}(z_N\left(q\right)-z_{N})-\lambda_1 D(z_N\left(q\right)-z_{N})+\lambda_2 (z_N\left(q\right)-z_{N}))\\
&=(z_N\left(q\right)-z_{N},\mathcal{L}_{\theta}^{\alpha}(u_N-u_N(q))+\lambda_1 D(u_N-u_N(q))+\lambda_2(u_N-u_N(q)))\\
&=-(u_N-u_N(q),u_N-u_N(q))\leq 0.
\end{align*}
 So we have
 \begin{align*}
 \gamma\left\|q-q_N\right\|^2&\leq\int_{\Omega}\left(z-z_N\left(q\right)\right)\left(q_{N}-q\right)dx\\
 &\leq \|z-z_N(q)\|_{\omega^{-\sigma^*,-\sigma}}\|q-q_N\|_{\omega^{\sigma^*,\sigma}}\\
 &\leq \|z-z_N(q)\|_{\omega^{-\sigma^*,-\sigma}}\|q-q_N\|.
 \end{align*}
Considering the above inequality and  $z-z_N(q)=z-z_N(u)+z_N(u)-z_N(q)$,  we get
 \begin{align}
 \left\|q-q_N\right\|\leq C\left(\|z-z_N(u)\|_{\omega^{-\sigma^*,-\sigma}}+\|z_N(u)-z_N(q)\|_{\omega^{-\sigma^*,-\sigma}}\right).\label{qerror}
 \end{align}

Subtracting \eqref{aub} from \eqref{auc}, it follows
$$b(z_N(u)-z_N(q),w_N)=(u-u_N(q),w_N), \forall w_N\in W_N.$$
Similar to the above derivation, by using \eqref{ues}, \eqref{infsup2} and Theorem \ref{regularity},  we have
\begin{align}
\|z_N(u)-z_N(q)\|_{\omega^{-\sigma^*,-\sigma}}\leq C\|u-u_N(q)\|_{\omega^{-\sigma,-\sigma^*}}\leq CN^{-m}.\label{zNuq}
\end{align}
Substituting \eqref{e3} and \eqref{zNuq} into \eqref{qerror}, it is obtained that
\begin{align}
\|q-q_N\|\leq CN^{-m}.\label{qerror1}
\end{align}
Combining the estimates \eqref{uerror}, \eqref{zerror} and\eqref{qerror1} lead us to attain the conclusion \eqref{conver}.
\end{proof}

\begin{section}{Numerical experiments}\label{section5}
In this section, we present a fast iteration algorithm for solving the discrete optimality conditions \eqref{adjoint} produced by the spectral Petrov-Galerkin method and give some  numerical examples to
verify the theoretical results.
\begin{subsection}{Numerical implementation}
Substituting $u_N=\omega^{\sigma,\sigma^*}\sum\limits^{N}_{n=0}\hat{u}_nQ_n^{\sigma,\sigma^*}$, $z_N=\omega^{\sigma^*,\sigma}\sum\limits^{N}_{n=0}\hat{z}_nQ_n^{\sigma^*,\sigma}$  into the discrete optimality conditions
\eqref{adjoint}  and taking $v_N=Q_m^{\sigma^*,\sigma},$ $w_N=Q_m^{\sigma,\sigma^*}$, $m=0,1,\cdots,N$, we obtain the following linear system by using \eqref{ptzz} and \eqref{ind}
\begin{subequations}\label{matrix}
\begin{numcases}
{}(S- \lambda_1 D+ \lambda_2 M)U=F,\label{matrixa}\\
(S+ \lambda_1\widehat{D}+\lambda_2 M^T)Z=G,\label{matrixb}\\
\gamma q_N=\max\{0,\hat{z}_0\, h_0^{\sigma^*,\sigma}\}-z_N.\label{matrixc}
\end{numcases}
\end{subequations}
where
\begin{align*}
S&=\diag(\lambda^{\alpha}_{1-\theta,0}\ h_0^{\sigma,\sigma^*},\lambda^{\alpha}_{1-\theta,1}\ h_1^{\sigma,\sigma^*},\cdots,\lambda^{\alpha}_{1-\theta,N}\ h_N^{\sigma,\sigma^*}),\\
U&=(\hat{u}_0,\hat{u}_1,\cdots,\hat{u}_N)^T,\ Z=(\hat{z}_0,\hat{z}_1,\cdots,\hat{z}_N)^T,\\
F&=(f_0,f_1,\cdots,f_N)^T,\   G=(g_0,g_1,\cdots,g_N)^T,
\end{align*}
  in which $f_m=(f+q_N,Q_m^{\sigma^*,\sigma})_{\omega^{\sigma^*,\sigma}}$,  $g_m=(u_N-u_d,Q_m^{\sigma,\sigma^*})_{\omega^{\sigma,\sigma^*}}$, $m=0,1,\cdots,N$. Moreover, $M$, $D$ and $\widehat{D}$ are the $(N+1)\times(N+1)$
  matrices with entries
\begin{eqnarray*}\begin{aligned}
M_{mn}&=(\omega^{\alpha,\alpha}Q_n^{\sigma,\sigma^*},Q_m^{\sigma^*,\sigma}),\\
D_{mn}&=-(m+1)(\omega^{\alpha-1,\alpha-1}Q_n^{\sigma,\sigma^*},Q_{m+1}^{\sigma^*-1,\sigma-1}),\\
\widehat{D}_{mn}&=-(m+1)(\omega^{\alpha-1,\alpha-1}Q_n^{\sigma^*,\sigma},Q_{m+1}^{\sigma-1,\sigma^*-1}).
\end{aligned}\end{eqnarray*}
The above linear optimization system \eqref{matrix} can be solved by the projected gradient method, see \cite{hinze2009, li2019, wang2021} for more details, which requires $O(N^2)$ storage and $O(N^3)$
computational complexity.
To improve the computational complexity, following the idea in \cite{hao2021} we present a fast projected gradient method with linear  storage $O(N)$  and quasilinear computational cost $O(N\log^2N)$.

Denote $A=S- \lambda_1 D+ \lambda_2 M$, $B=S+\lambda_1 \widehat{D}+ \lambda_2 M^T$, we use the fixed-point iteration
\begin{align}
U^{m+1}&=U^m+P^{-1}(F-AU^m),\label{AU}\\
Z^{m+1}&=Z^m+\widehat{P}^{-1}(G-BZ^m),\label{BZ}
\end{align}
where the preconditioners $P=S- \lambda_1 K+ \lambda_2 Q $ and $\widehat{P}=S+\lambda_1\widehat{K}+\lambda_2 Q $ with
\begin{equation}\label{eq:Q}
	Q=\diag(h_0^{\alpha,\alpha},h_1^{\alpha,\alpha},\cdots,h_N^{\alpha,\alpha} ),
	\end{equation}
 $K$ and $\widehat{K}$ are two tridiagonal matrices with
 \begin{eqnarray*}\begin{aligned}
 K_{n,n+1}&=D_{n,n+1}, \ K_{n+1,n}=D_{n+1,n},\\
 \widehat{K}_{n,n+1}&=\widehat{D}_{n,n+1}, \ \widehat{K}_{n+1,n}=\widehat{D}_{n+1,n},\ n=0,1,\cdots,N,
 \end{aligned}\end{eqnarray*}
  and other entries are zero. In each iteration, without forming the matrices $A$ and $B$, we compute the matrix-vector products $AU$ and $BZ$ by using the fast polynomial transforms and the fast matrix-vector product
  \cite{town2018} for Toeplitz-dot-Hankel matrices (can be written as a Hadamard product of a Toeplitz and Hankel matrix).

Let ${\bf{{Q}}}_k^{\gamma,\beta}=(Q_0^{\gamma,\beta},Q_1^{\gamma,\beta},\cdots,Q_k^{\gamma,\beta})^T$, $k\geq0$. we have the following transformation formulas \cite{askey1975} of Jacobi ploynomials
\begin{equation}\label{eq:ctransform}
	{\bf{{Q}}}_k^{\gamma,\beta}=C_k^{\gamma\rightarrow\sigma,\beta}{\bf{{Q}}}_k^{\sigma,\beta},\ {\bf{{Q}}}_k^{\sigma,\delta}=C_k^{\sigma,\delta\rightarrow\beta}{\bf{{Q}}}_k^{\sigma,\beta}.
	\end{equation}
 Here $C_k^{\gamma\rightarrow\sigma,\beta}$ and $C_k^{\sigma,\delta\rightarrow\beta}$ are $(k+1)\times(k+1)$ lower triangular conversion matrices and can be decomposed into 
 matrices
 \begin{eqnarray}\label{TH}\begin{aligned}
 C_k^{\gamma\rightarrow\sigma,\beta}=\overline{D}_1(T_1\circ H_1)\overline{D}_2,\ C_k^{\sigma,\delta\rightarrow\beta}=\overline{D}_3(T_2\circ H_2)\overline{D}_4,
 \end{aligned}\end{eqnarray}
 where $\overline{D}_i$, $i=1,2,3,4$, are diagonal matrices, $T_i$, $i=1,2$ are lower triangular Toeplitz matrices, $H_i$, $i=1,2$ are Hankel matrices and `$\circ$' is the Hadamard matrix product. We refer to
  \cite{hao2021} for explicit entries of these matrices.

 Denote $u_N=(1-x)^{\sigma} x^{\sigma^*}\tilde{u}_N$, by using \eqref{TH} we have
\begin{equation}\label{eq:un}
	\tilde{u}_N=\sum\limits^{N}_{n=0}\hat{u}_nQ_n^{\sigma,\sigma^*}=({\bf{Q}}_N^{\sigma,\sigma^*})^TU=({\bf{{Q}}}_N^{\alpha,\alpha})^T(C_N^{\alpha,\sigma^*\rightarrow\alpha})^T(C_N^{\sigma\rightarrow\alpha,\sigma^*})^T \
 U:=({\bf{Q}}_N^{\alpha,\alpha})^T\ U^{\alpha},
 \end{equation}
where $U^{\alpha}=(\hat{u}^{\alpha}_0,\hat{u}^{\alpha}_1,\cdots,\hat{u}^{\alpha}_N)^T$ can be obtained by the fast Jacobi-to-Jacobi polynomial transform \cite{town2018}. Note that $({\bf{Q}}_N^{\alpha,\alpha})^T\ U^{\alpha}=\sum\limits^{N}_{n=0}\hat{u}^{\alpha}_nQ_n^{\alpha,\alpha}$.

In the following we show the algorithm for computing $AU=(S-\lambda_1 D+ \lambda_2 M)U$. Note that $S$ is a diagonal matrix, so we just need consider $DU$ and $MU$.
For the fast matrix-vector product of $MU$, we have
 \begin{eqnarray}\label{MU}\begin{aligned}
 MU&=\int_{\Omega}\omega^{\alpha,\alpha}{\bf{Q}}_N^{\sigma^*,\sigma}({\bf{Q}}_N^{\sigma,\sigma^*})^Tdx\ U  \\
 &=\int_{\Omega}\omega^{\alpha,\alpha}C_N^{\sigma^*,\sigma\rightarrow\alpha}C_N^{\sigma^*\rightarrow\alpha,\alpha}{\bf{Q}}_N^{\alpha,\alpha}(C_N^{\sigma\rightarrow\alpha,\sigma^*}C_N^{\alpha,\sigma^*\rightarrow\alpha}{\bf{Q}}_N^{\alpha,\alpha})^Tdx\
 U\\
  &=C_N^{\sigma^*,\sigma\rightarrow\alpha}C_N^{\sigma^*\rightarrow\alpha,\alpha}\ Q\ (C_N^{\alpha,\sigma^*\rightarrow\alpha})^T(C_N^{\sigma\rightarrow\alpha,\sigma^*})^T \ U\\
  &=C_N^{\sigma^*,\sigma\rightarrow\alpha}C_N^{\sigma^*\rightarrow\alpha,\alpha}\ Q\ U^{\alpha}.
 \end{aligned}\end{eqnarray}
 By \eqref{TH}, $C^{\sigma^*,\sigma\rightarrow\alpha}_N$ and $C^{\sigma^*\rightarrow\alpha,\alpha}_N$ both are Toeplitz-dot-Hankel matrices, so  $MU$ can be computed without forming a matrix by using the fast matrix-vector
 product \cite{town2018}.

To fulfill the fast matrix-vector product for $DU$, we denote
 \begin{align}
  &\widehat{Q}=\diag(h_0^{\alpha,\alpha},h_1^{\alpha,\alpha},\cdots,h_N^{\alpha,\alpha},0),\label{eq:qhat} \\
  &U^{\alpha-1}=(\hat{u}^{\alpha-1}_0,\hat{u}^{\alpha-1}_1,\cdots,\hat{u}^{\alpha-1}_N)^T,\label{eq:ualpha-1}\\
  &\Lambda=-\diag(1,2,\cdots,N+1),\label{eq:Lambda}
 \end{align}
 and  introduce a $(N+2)\times (N+1)$ matrix $W$ defined as
 \begin{align*}
 W&= \int_{\Omega}\omega^{\alpha-1,\alpha-1}{\bf{Q}}_{N+1}^{\sigma^*-1,\sigma-1}({\bf{Q}}_N^{\sigma,\sigma^*})^Tdx\\
 &=
 \int_{\Omega}\omega^{\alpha-1,\alpha-1}C_{N+1}^{\sigma^*-1,\sigma-1\rightarrow\alpha-1}C_{N+1}^{\sigma^*-1\rightarrow\alpha-1,\alpha-1}{\bf{Q}}_{N+1}^{\alpha-1,\alpha-1}(C_N^{\sigma\rightarrow\alpha-1,\sigma^*}C_N^{\alpha-1,\sigma^*\rightarrow\alpha-1}{\bf{Q}}_N^{\alpha-1,\alpha-1})^Tdx\\
 &=C_{N+1}^{\sigma^*-1,\sigma-1\rightarrow\alpha-1}C_{N+1}^{\sigma^*-1\rightarrow\alpha-1,\alpha-1}\widehat{Q}(C_N^{\alpha-1,\sigma^*\rightarrow\alpha-1})^T(C_N^{\sigma\rightarrow\alpha-1,\sigma^*})^T.
 \end{align*}
   Similarly, we have
   \begin{eqnarray*}\begin{aligned}
   \tilde{u}_N&=\sum\limits^{N}_{n=0}\hat{u}_nQ_n^{\sigma,\sigma^*}=({\bf{Q}}_N^{\sigma,\sigma^*})^TU=({\bf{Q}}_N^{\alpha-1,\alpha-1})^T(C_N^{\alpha-1,\sigma^*\rightarrow\alpha-1})^T(C_N^{\sigma\rightarrow\alpha-1,\sigma^*})^TU\\
   &:=({\bf{Q}}_N^{\alpha-1,\alpha-1})^T\ U^{\alpha-1}=\sum\limits^{N}_{n=0}\hat{u}^{\alpha-1}_n\ Q_n^{\alpha-1,\alpha-1}.
   \end{aligned}\end{eqnarray*}
  Therefore, $WU$ also can be obtained without forming a matrix by using  the fast matrix-vector product combined with the fast Jacobi-to-Jacobi polynomial transform.
  For the advection term, we can write
 \begin{align}\label{eq:wu}
 DU=\Lambda\widehat{W}U,
 \end{align}
 where $\widehat{W}$ denotes the sub-matrix of $W$ formed by removing its first row, and $\widehat{W}U$ can be obtained conveniently by removing the first element of vector $WU$.

 In summary, for the fast calculation of $AU$, we give
 \begin{equation}\label{eq:AU}
 AU=SU-\lambda_1 \Lambda\widehat{W}U+ \lambda_2 C_N^{\sigma^*,\sigma\rightarrow\alpha}C_N^{\sigma^*\rightarrow\alpha,\alpha}\ Q\ U^{\alpha},
 \end{equation}
where $\Lambda$, $Q$, $ C_N^{\sigma^*,\sigma\rightarrow\alpha}$, $C_N^{\sigma^*\rightarrow\alpha,\alpha}$, $U^{\alpha}$ and $\widehat{W}U$ are given above, see \eqref{eq:Lambda}, \eqref{eq:Q}, \eqref{eq:ctransform}, \eqref{eq:un} and \eqref{eq:wu}, respectively.

Fast calculation of matrix-vector product $BZ$ in the fixed-point iteration \eqref{BZ}  can  be fulfilled by an
 analogous way. We omit the details here for brevity.  Moreover, the right hand side $F$ can also be obtained in a similar manner,  see \eqref{f} in next subsection  for more details.

 The fast projected gradient method can be summarized as follows:
 \begin{algorithm}[H]
\caption{Fast projected gradient algorithm}

\label{alg:SA}

\begin{algorithmic}[1]

\STATE  Solving the system \eqref{matrix} with  $N=8$ by the general projected gradient method combined with a direct solver to get the initial guess
$(U^0,Z^0,q_N)$.

\STATE Setting $error=1$, and given the tolerance $\epsilon$.
\WHILE{$error>\epsilon$}
     \STATE Solving the state equation in \eqref{matrixa} by the fixed-point iteration \eqref{AU}:
     $$U^{m+1}=U^m+P^{-1}(F-AU^m),$$
    where $AU^m$ can be computed by \eqref{eq:AU}.
     \STATE Solving the adjoint state equation  in \eqref{matrixb} by the fixed-point iteration \eqref{BZ}:
     $$Z^{m+1}=Z^m+\widehat{P}^{-1}(G-BZ^m),$$
     where $BZ^m$ can be computed in a similar manner to $AU^m$.
     \STATE Computing the control variable $q_N$:\\
$$q_N^{\text{new}} =\max\{0, \frac{\hat{z}_0\,h_0^{\sigma^*,\sigma}}{\gamma}\}-\frac{z_N}{\gamma}.$$

 \STATE Calculate the error by
 $error:=norm(q_N-q_N^{\text{new}},inf)/norm(q_N,inf).$
 \STATE
 Update the control variable $q_N=q_N^{\text{new}}$
\ENDWHILE              
\end{algorithmic}
\end{algorithm}
 \end{subsection}

\begin{subsection}{Numerical results}

In the optimal control problem \eqref{FOCPa1}-\eqref{FOCPb1} with $\gamma=1$, we take
\begin{align}
f=(1-x)^{\beta}x^{\beta}\sin x,\ u_d=(1-x)^{\beta}x^{\beta}\cos x.\label{fud}
\end{align}

The right hand side $f_m$ and $g_m$ can be computed as $(MU)_m$. Specifically, we evaluate $\sin x$ and $\cos x$ at the Chebyshev collocation points $x_i$ ($1\leq i \leq M$, $M\geq N$), and by the inverse fast Chebyshev
transform \cite{town2018} we can obtain the approximations
\begin{equation}\label{f}
f\approx\omega^{\beta,\beta}\sum\limits_{n=0}^{M}\hat{f}^{-1/2}Q_n^{-1/2,-1/2},\ u_d\approx\omega^{\beta,\beta}\sum\limits_{n=0}^{M}\hat{g}^{-1/2}Q_n^{-1/2,-1/2}.
\end{equation}
Then invoking the fast Jacobi-to-Jacobi polynomial transform separately on $f$ and $q_N$ in $f_m$ ($u_d$ and $u_N$ in $g_m$), the right hand terms $F$ and $G$ in \eqref{matrix} can be observed in an
analogous treatment as in \eqref{MU}.

In the computation, we take $\lambda_1=\lambda_2=1$ and  measure the errors in the following sense:
\begin{equation*}\begin{aligned}
E_N^{a,b}(p)=\frac{\|p_N-p_{N_r}\|_{\omega^{a,b}}}{\|p_{N_r}\|_{\omega^{a,b}}},
\end{aligned}\end{equation*}
where $a,\,b>-1$, $p_N$ is the numerical solution and $p_{N_r}$ is the reference solution computed by the same solver but with a very fine resolution, $N_r=2^{14}$.  In Table \ref{Table1} cited from \cite{hao1}, we present the numerical
values of $(\sigma,\sigma^*)$ corresponding to different $\alpha$ and $\theta$, which were  calculated by using Newton's method with a tolerance $10^{-16}$.

\begin{example}\label{example1}
 Take $\beta=0$ in \eqref{fud} and $\theta=0.7$ in  \eqref{FOCPb1}. Note that $f\in H^{\infty}_{\omega^{\sigma^*,\sigma}}(\Omega)$, $u_d\in H^{\infty}_{\omega^{\sigma,\sigma^*}}(\Omega)$.
\end{example}

 As $f$ and $u_d$ are analytic, we have
 $$\omega^{-\sigma,-\sigma^*}u\in H^{2\alpha+\min\{\sigma,\sigma^*\}-1-\epsilon}_{\omega^{\sigma,\sigma^*}}(\Omega),\ \omega^{-\sigma^*,-\sigma}z\in H^{2\alpha+\min\{\sigma,\sigma^*\}-1-\epsilon}_{\omega^{\sigma^*,\sigma}}(\Omega)$$
 by Theorem \ref{regularity}. From Theorem \ref{converth}, the convergence order of  numerical solutions $(u_N,z_N,q_N)$ is expected to be $2\alpha+\min\{\sigma,\sigma^*\}-1-\epsilon$.
We test the  performance of the spectral Petrov-Galerkin method and the fast projected gradient method for different values of $\alpha$. In Table \ref{Table2}, we measure the  convergence orders and errors of
state $u_N$ in $\|\cdot\|_{\omega^{-\sigma,-\sigma^*}}$ norm,  adjoint state $z_N$ and  control $q_N$ in $\|\cdot\|_{\omega^{-\sigma^*,-\sigma}}$ norm. In Table \ref{Table3}, we show the convergence orders and errors of
numerical solutions in standard $L^2$-norm.

Data in Table \ref{Table2} show that the convergence orders of $u_N$ in $\|\cdot\|_{\omega^{-\sigma,-\sigma^*}}$ norm and $z_N$ in $\|\cdot\|_{\omega^{-\sigma^*,-\sigma}}$ norm are
$2\alpha+\min\{\sigma,\sigma^*\}-1-\epsilon$, as expected from Theorem \ref{converth}. In Table \ref{Table3}, we observe that the convergence orders of  control $q_N$ in standard $L^2$-norm are better than expected
$2\alpha+\min\{\sigma,\sigma^*\}-1-\epsilon$, while its convergence orders in $\|\cdot\|_{\omega^{-\sigma^*,-\sigma}}$ norm are numerically close to $2\alpha+\min\{\sigma,\sigma^*\}-1-\epsilon$, see Table \ref{Table2}.
Moreover, numerical results in Tables \ref{Table2} and \ref{Table3} also show that the convergence orders and accuracy of  numerical solutions $(u_N,z_N,q_N)$ in standard  $L^2$-norm are higher than these in weighted
$L^2$-norm with negative index, which numerically indicate that it is meaningful to consider the spectral Petrov-Galerkin method for the optimal control problem in weighted Sobolev space.

In the last two rows of Tables, we list the iterative numbers and CPU time required for the fast projected gradient algorithm with different values of $N$ and $\alpha$. We observe that the CPU time increases roughly as
$O(N\log^2N)$ and the iteration numbers are independent on truncation numbers $N$. However, the iteration numbers decrease with $\alpha$: when $\alpha=1.2$, iteration number is $51$ while the number is 9 for $\alpha=1.8$,
which suggests the need of searching for better preconditioner $P$.
\begin{table}[htbp]
\scriptsize
\caption{The numerical values of $(\sigma,\sigma^*)$ corresponding to different $\alpha$ and $\theta$ (cited from \cite{hao1}).}
\label{Table1}
\centering
\begin{tabular}{ccccc}
\toprule
 $\theta$ &$\alpha$=1.2 & $\alpha$=1.4 & $\alpha$=1.6 & $\alpha$=1.8\\
 \midrule
 0.5      & (0.6000, 0.6000) &(0.7000, 0.7000)& (0.8000, 0.8000)& (0.9000, 0.9000)\\
 \specialrule{0em}{2pt}{2pt}
 0.7      & (0.8829, 0.3171) &(0.8602, 0.5398) &(0.8900, 0.7100) &(0.9411, 0.8589)  \\
 \specialrule{0em}{2pt}{2pt}
 1.0      & (1.0000, 0.2000)& (1.0000, 0.4000) &(1.0000, 0.6000) &(1.0000, 0.8000)\\
 \bottomrule
\end{tabular}
\end{table}
\begin{table}[htbp]
\scriptsize
\caption{Convergence orders and errors of the spectral Petrov-Galerkin method  with $\theta=0.7$, $f=\sin x$ and $u_d=\cos x$ in weighted norm. The expected convergence order is $2\alpha+\min\{\sigma,\sigma^*\}-1-\epsilon$  (Theorems
\ref{regularity} and \ref{converth}).}
\label{Table2}
\centering
\scalebox{0.92}{\begin{tabular}{cccccccccc}
\toprule
$\alpha$ & $N$ &$E^{-\sigma,-\sigma^*}_N(u)$ &  order& $E^{-\sigma^*,-\sigma}_N(z)$ &   order &$E^{-\sigma^*,-\sigma}_N(q)$ & order& iter &CPU(s) \\
\midrule
\multirow{4}{*}{1.2}
& 128       &  3.74e-05     &      *     &    5.61e-05     &     *     &   5.61e-05     &     *    &   51     &     6.70    \\
     & 256       &  7.46e-06    &    2.33      &  1.15e-05    &    2.28       &  1.15e-05    &    2.28      &  51     &     7.46    \\
     & 512       &  1.55e-06    &    2.27      &  2.51e-06    &    2.20       &  2.51e-06    &    2.20      &  51     &    10.16    \\
     & 1024       &  3.38e-07    &    2.20      &  5.83e-07    &    2.10       &  5.83e-07    &    2.10      &  51     &    27.25    \\
\hline
\multicolumn{2}{c}{Expected order}&  & 1.72 & &1.72 & & & &  \\
\hline

\multirow{4}{*}{1.4}

& 128       &  1.56e-06     &      *     &    2.41e-06     &     *     &   2.41e-06     &     *    &   15     &     1.17    \\
    & 256       &  3.05e-07    &    2.35      &  4.78e-07    &    2.33       &  4.78e-07    &    2.33      &  15     &     1.50    \\
    & 512       &  6.00e-08    &    2.35      &  9.46e-08    &    2.34       &  9.46e-08    &    2.34      &  15     &     2.14    \\
    & 1024       &  1.18e-08    &    2.35      &  1.87e-08    &    2.34       &  1.87e-08    &    2.34      &  15     &     6.25    \\
 \hline
\multicolumn{2}{c}{Expected order}&  & 2.34 & & 2.34& & & &  \\
\hline

\multirow{4}{*}{1.6}

& 128       &  8.47e-08     &      *     &    1.09e-07     &     *     &   1.09e-07     &     *    &   12     &     0.68    \\
    & 256       &  1.15e-08    &    2.89      &  1.15e-05    &    2.88       &  1.48e-08    &    2.88      &  12     &     0.97    \\
    & 512       &  1.54e-09    &    2.90      &  2.51e-06    &    2.90       &  1.99e-09    &    2.90      &  12     &     1.34    \\
    & 1024       &  2.05e-10    &    2.91      &  5.83e-07    &    2.90       &  2.66e-10    &    2.90      &  12     &     4.08    \\
 \hline
\multicolumn{2}{c}{Expected order}&  & 2.91 & &2.91 & & & &  \\

\hline

\multirow{4}{*}{1.8}

& 128       &  3.02e-09     &      *     &    3.29e-09     &     *     &   3.29e-09     &     *    &   9     &     0.48    \\
    & 256       &  2.81e-10    &    3.42      &  3.07e-10    &    3.42       &  3.07e-10    &    3.42      &  9     &     0.69    \\
    & 512       &  2.58e-11    &    3.44      &  2.83e-11    &    3.44       &  2.83e-11    &    3.44      &  9     &     1.01    \\
    & 1024       &  2.37e-12    &    3.45      &  2.60e-12    &    3.44       &  2.60e-12    &    3.44      &  9     &     3.04    \\
 \hline
\multicolumn{2}{c}{Expected order}&  & 3.46 & & 3.46& & & &  \\
 \bottomrule
\end{tabular}}
\end{table}

\begin{table}[htbp]
\scriptsize
\caption{Convergence orders and errors of the spectral Petrov-Galerkin method  with $\theta=0.7$, $f=\sin x$ and $u_d=\cos x$ in standard $L^2$-norm (higher than that in Table \ref{Table2}).}
\label{Table3}
\centering
\scalebox{0.92}{\begin{tabular}{cccccccccc}
\toprule
$\alpha$ & $N$ &$E^{0,0}_N(u)$ &  order& $E^{0,0}_N(z)$ &   order &$E^{0,0}_N(q)$ & order& iter &CPU(s) \\
\midrule
\multirow{4}{*}{1.2}
& 64       &  8.44e-05     &      *     &    1.35e-04     &     *     &   1.35e-04     &     *    &   51     &     3.41    \\
     & 128       &  1.14e-05    &    2.89      &  1.90e-05    &    2.82       &  1.90e-05    &    2.82      &  51     &     6.68    \\
     & 256       &  1.62e-06    &    2.82      &  2.92e-06    &    2.71       &  2.92e-06    &    2.71      &  51     &     7.52    \\
     & 512       &  2.57e-07    &    2.65      &  5.19e-07    &    2.49       &  5.19e-07    &    2.49      &  51     &    10.10    \\

\hline

\multirow{4}{*}{1.4}

 & 64       &  2.45e-06     &      *     &    3.79e-06     &     *     &   3.79e-06     &     *    &   15     &     0.55    \\
    & 128       &  3.72e-07    &    2.72      &  5.86e-07    &    2.69       &  5.86e-07    &    2.69      &  15     &     1.17    \\
    & 256       &  5.58e-08    &    2.74      &  8.87e-08    &    2.73       &  8.87e-08    &    2.73      &  15     &     1.39    \\
    & 512       &  8.25e-09    &    2.76      &  1.32e-08    &    2.75       &  1.32e-08    &    2.75      &  15     &     2.10    \\

\hline

\multirow{4}{*}{1.6}

& 64       &  1.74e-07     &      *     &    2.24e-07     &     *     &   2.24e-07     &     *    &   12     &     0.32    \\
    & 128       &  1.78e-08    &    3.29      &  1.90e-05    &    3.28       &  2.31e-08    &    3.28      &  12     &     0.70    \\
    & 256       &  1.77e-09    &    3.33      &  2.92e-06    &    3.33       &  2.29e-09    &    3.33      &  12     &     0.90    \\
    & 512       &  1.72e-10    &    3.36      &  5.19e-07    &    3.36       &  2.23e-10    &    3.36      &  12     &     1.33    \\

\hline

\multirow{4}{*}{1.8}

 & 64       &  8.56e-09     &      *     &    9.35e-09     &     *     &   9.35e-09     &     *    &   9     &     0.23    \\
    & 128       &  5.98e-10    &    3.84      &  6.54e-10    &    3.84       &  6.54e-10    &    3.84      &  9     &     0.49    \\
    & 256       &  4.02e-11    &    3.90      &  4.39e-11    &    3.90       &  4.39e-11    &    3.90      &  9     &     0.66    \\
    & 512       &  2.65e-12    &    3.92      &  3.15e-12    &    3.80       &  3.15e-12    &    3.80      &  9     &     1.01    \\

 \bottomrule
\end{tabular}}
\end{table}
\begin{example}\label{example2}
 Take $\beta=-0.4$ in \eqref{fud} and  $\theta=0.5,\,0.7,\,1$ in \eqref{FOCPb1}.   Note that $$f\in H^{2\beta+\min\{\sigma,\sigma^*\}+1-\epsilon}_{\omega^{\sigma^*,\sigma}}(\Omega),\, u_d\in
 H^{2\beta+\min\{\sigma,\sigma^*\}+1-\epsilon}_{\omega^{\sigma,\sigma^*}}(\Omega),$$ by Lemma B.4 in \cite{hao1}.
\end{example}

By Theorem \ref{regularity}, we have
\begin{align*}
\omega^{-\sigma,-\sigma^*}u\in H^{\min\{2\beta+\alpha+1,2\alpha-1\}+\min\{\sigma,\sigma^*\}-\epsilon}_{\omega^{\sigma,\sigma^*}}(\Omega),\\
\text{and}\ \ \omega^{-\sigma^*,-\sigma}z\in
H^{\min\{2\beta+\alpha+1,2\alpha-1\}+\min\{\sigma,\sigma^*\}-\epsilon}_{\omega^{\sigma^*,\sigma}}(\Omega).
\end{align*}
 According to Theorem \ref{converth}, we  expect that the convergence order of state $u_N$ in
$\|\cdot\|_{\omega^{-\sigma,-\sigma^*}}$ norm, adjoint state $z_N$ in $\|\cdot\|_{\omega^{-\sigma^*,-\sigma}}$ norm and control $q_N$ in standard $L^2$-norm will be
$\min\{2\beta+\alpha+1,2\alpha-1\}+\min\{\sigma,\sigma^*\}-\epsilon$. In Tables \ref{Table4}-\ref{Table6}, we  test the convergence orders of numerical solutions and examine the performance of the fast projected gradient
algorithm for $\theta=0.5$, $0.7$ and $1$, respectively.

The numerical results confirm the theoretically predicted orders for state $u_N$ and adjoint state $z_N$ with different values of $\theta$ and  $\alpha$.  It is shown that the $L^2$-errors of  control $q_N$ have higher
convergence order than expected  and its convergence order in $\|\cdot\|_{\omega^{-\sigma^*,-\sigma}}$ norm is close to $\min\{2\beta+\alpha+1,2\alpha-1\}+\min\{\sigma,\sigma^*\}-\epsilon$.

 Note that when $\theta=0.5$, $\sigma=\sigma^*=\alpha/2$, the two-side fractional derivative operator $L^{\alpha}_{\theta}$ is equivalent to the \emph{integral} fractional Laplacian in one-dimensional and the spectral Petrov-Galerkin
 method is reduced to the spectral Galerkin method in \cite{wang2021} for the optimal control problem \eqref{FOCPa1}-\eqref{admissibleset} with $\theta=0.5$. Table \ref{Table4} shows that our fast projected gradient
 algorithm is several times faster than the fast algorithm presented in \cite{wang2021} for the same problem. Moreover, from Tables \ref{Table4}-\ref{Table6}, we find that the iterative numbers and CPU time were greatly
 reduced for $\theta=0.5$ and $1$, which inspired us to search for a better preconditioner in the future work.
\begin{table}[H]
\scriptsize
\caption{Convergence orders and errors of the spectral Petrov-Galerkin method  with $\theta=0.5$, $f=(1-x)^{-0.4}x^{-0.4}\sin x$ and $u_d=(1-x)^{-0.4}x^{-0.4}\cos x$. The expected convergence order is
$2\beta+3\alpha/2+1-\epsilon$  (Theorems \ref{regularity} and \ref{converth}).}
\label{Table4}
\centering
\scalebox{0.92}{
\begin{tabular}{cccccccccccc}
\toprule
$\alpha$ & $N$ &$E^{-\frac{\alpha}{2},-\frac{\alpha}{2}}_N(u)$ &  order& $E^{-\frac{\alpha}{2},-\frac{\alpha}{2}}_N(z)$ &   order &$E^{-\frac{\alpha}{2},-\frac{\alpha}{2}}_N(q)$ & order& $E^{0,0}_N(q)$ & order& iter &CPU(s)
\\
\midrule
\multirow{4}{*}{1.2}
 & 128       &  1.04e-03     &      *     &    7.17e-04     &     *     &   7.17e-04     &     *    &   2.33e-04     &     *    &   33     &     3.52    \\
     & 256       &  2.96e-04    &    1.81      &  2.05e-04    &    1.81       &  2.05e-04    &    1.81     &  5.01e-05    &    2.22   &  33     &     4.01    \\
     & 512       &  8.11e-05    &    1.87      &  5.60e-05    &    1.87       &  5.60e-05    &    1.87     &  1.00e-05    &    2.32     &  33     &     5.64    \\
     & 1024       &  2.16e-05    &    1.91      &  1.49e-05    &    1.91       &  1.49e-05    &    1.91      &  1.91e-06    &    2.39 &  33     &    15.27    \\
\hline
\multicolumn{2}{c}{Expected order}&  & 2.0 & &2.0 & & & & 2.0&& \\
\hline

\multirow{4}{*}{1.4}

  & 128       &  3.17e-06     &      *     &    6.10e-06     &     *     &   6.10e-06     &     *   &   1.07e-06     &     *  &   14     &     0.99    \\
    & 256       &  6.51e-07    &    2.28      &  1.31e-06    &    2.21       &  1.31e-06    &    2.21     &  1.70e-07    &    2.66   &  14     &     1.34    \\
    & 512       &  1.35e-07    &    2.27      &  2.80e-07    &    2.23       &  2.80e-07    &    2.23      &  2.63e-08    &    2.69 &  14     &     1.81    \\
    & 1024       &  2.80e-08    &    2.26      &  5.90e-08    &    2.25       &  5.90e-08    &    2.25     &  4.03e-09    &    2.71   &  14     &     5.44    \\
 \hline
\multicolumn{2}{c}{Expected order}&  & 2.3 & & 2.3& & & &2.3 && \\
\hline

\multirow{4}{*}{1.6}

 & 128       &  1.09e-06     &      *     &    1.26e-06     &     *     &   1.26e-06     &     *   &   2.22e-07     &     *   &   11     &     0.67    \\
    & 256       &  1.85e-07    &    2.55      &  2.05e-04    &    2.55       &  2.16e-07    &    2.55    &  2.77e-08    &    3.01   &  11     &     0.96    \\
    & 512       &  3.12e-08    &    2.57      &  5.60e-05    &    2.57       &  3.64e-08    &    2.57    &  3.36e-09    &    3.04   &  11     &     1.36    \\
    & 1024       &  5.22e-09    &    2.58      &  1.49e-05    &    2.58       &  6.10e-09    &    2.58    &  4.03e-10    &    3.06  &  11     &     4.07    \\
 \hline
\multicolumn{2}{c}{Expected order}&  & 2.6 & &2.6 & & & & 2.6&& \\

\hline

\multirow{4}{*}{1.8}

 & 128       &  3.40e-07     &      *     &    2.76e-07     &     *     &   2.76e-07     &     *   &   4.83e-08     &     *  &   9     &     0.47    \\
    & 256       &  4.65e-08    &    2.87      &  3.78e-08    &    2.87       &  3.78e-08    &    2.87     &  4.76e-09    &    3.34   &  9     &     0.66    \\
    & 512       &  6.29e-09    &    2.89      &  5.13e-09    &    2.88       &  5.13e-09    &    2.88       &  4.61e-10    &    3.37 &  9     &     0.98    \\
    & 1024       &  8.47e-10    &    2.89      &  6.92e-10    &    2.89       &  6.92e-10    &    2.89      &  4.42e-11    &    3.38 &  9     &     2.99    \\

 \hline
\multicolumn{2}{c}{Expected order}&  & 2.9 & & 2.9& & & &2.9 && \\
 \bottomrule
\end{tabular}}
\end{table}

\begin{table}[H]
\scriptsize
\caption{Convergence orders and errors of the spectral Petrov-Galerkin method  with $\theta=0.7$, $f=(1-x)^{-0.4}x^{-0.4}\sin x$ and $u_d=(1-x)^{-0.4}x^{-0.4}\cos x$. The expected convergence order is
$2\beta+\min\{\sigma,\sigma^*\}+\alpha+1-\epsilon$  (Theorems \ref{regularity} and \ref{converth}).}
\label{Table5}
\centering
\scalebox{0.92}{\begin{tabular}{cccccccccccc}
\toprule
$\alpha$ & $N$ &$E^{-\sigma,-\sigma^*}_N(u)$ &  order& $E^{-\sigma^*,-\sigma}_N(z)$ &   order &$E^{-\sigma^*,-\sigma}_N(q)$ & order&  $E^{0,0}_N(q)$ & order&iter &CPU(s) \\
\midrule
\multirow{4}{*}{1.2}
& 128       &  2.43e-05     &      *     &    1.65e-04     &     *     &   1.65e-04     &     *   &   5.10e-05     &     *  &   51     &     6.77    \\
     & 256       &  5.44e-06    &    2.16      &  4.18e-05    &    1.98       &  4.18e-05    &    1.98     &  9.57e-06    &    2.41  &  51     &     7.50    \\
     & 512       &  1.54e-06    &    1.83      &  1.11e-05    &    1.91       &  1.11e-05    &    1.91     &  1.97e-06    &    2.28 &  51     &    10.15    \\
     & 1024       &  4.77e-07    &    1.69      &  3.05e-06    &    1.86       &  3.05e-06    &    1.86     &  4.37e-07    &    2.17   &  51     &    27.55    \\
\hline
\multicolumn{2}{c}{Expected order}&  & 1.72 & &1.72 & & & & 1.72&& \\
\hline

\multirow{4}{*}{1.4}

 & 128       &  2.41e-06     &      *     &    2.49e-06     &     *     &   2.49e-06     &     *    &   4.51e-07     &     *  &   15     &     1.19    \\
    & 256       &  5.75e-07    &    2.07      &  5.50e-07    &    2.18       &  5.50e-07    &    2.18   &  7.38e-08    &    2.61     &  15     &     1.50    \\
    & 512       &  1.34e-07    &    2.10      &  1.22e-07    &    2.17       &  1.22e-07    &    2.17       &  1.21e-08    &    2.61  &  15     &     2.24    \\
    & 1024       &  3.10e-08    &    2.12      &  2.69e-08    &    2.18       &  2.69e-08    &    2.18      &  1.95e-09    &    2.63  &  15     &     6.33    \\
 \hline
\multicolumn{2}{c}{Expected order}&  & 2.14 & & 2.14& & & &2.14 & & \\
\hline

\multirow{4}{*}{1.6}

& 128       &  8.87e-07     &      *     &    6.82e-07     &     *     &   6.82e-07     &     *   &   1.38e-07     &     *  &   12     &     0.69    \\
    & 256       &  1.52e-07    &    2.54      &  4.18e-05    &    2.61       &  1.12e-07    &    2.61   &  1.66e-08    &    3.06    &  12     &     0.96    \\
    & 512       &  2.60e-08    &    2.55      &  1.11e-05    &    2.63       &  1.80e-08    &    2.63     &  1.94e-09    &    3.10  &  12     &     1.37    \\
    & 1024       &  4.44e-09    &    2.55      &  3.05e-06    &    2.65       &  2.87e-09    &    2.65     &  2.23e-10    &    3.12   &  12     &     4.14    \\
 \hline
\multicolumn{2}{c}{Expected order}&  & 2.51 & &2.51 & & & &2.14 & &\\

\hline

\multirow{4}{*}{1.8}

 & 128       &  3.14e-07     &      *     &    2.25e-07     &     *     &   2.25e-07     &     * &   4.13e-08     &     *     &   9     &     0.49    \\
    & 256       &  4.26e-08    &    2.88      &  3.01e-08    &    2.90       &  3.01e-08    &    2.90    &  3.98e-09    &    3.37  &  9     &     0.70    \\
    & 512       &  5.73e-09    &    2.89      &  3.99e-09    &    2.92       &  3.99e-09    &    2.92    &  3.77e-10    &    3.40    &  9     &     1.04    \\
    & 1024       &  7.69e-10    &    2.90      &  5.25e-10    &    2.93       &  5.25e-10    &    2.93    &  3.53e-11    &    3.42   &  9     &     3.07    \\
 \hline
\multicolumn{2}{c}{Expected order}&  & 2.86 & & 2.86& & & &2.86 & &\\
 \bottomrule
\end{tabular}}
\end{table}

\begin{table}[htbp]
\scriptsize
\caption{Convergence orders and errors of the spectral Petrov-Galerkin method  with $\theta=1$, $f=(1-x)^{-0.4}x^{-0.4}\sin x$ and $u_d=(1-x)^{-0.4}x^{-0.4}\cos x$. The expected convergence order is
$2\beta+\min\{\sigma,\sigma^*\}+\alpha+1-\epsilon$ (Theorems \ref{regularity} and \ref{converth}).}
\label{Table6}
\centering
\scalebox{0.92}{\begin{tabular}{cccccccccccc}
\toprule
$\alpha$ & $N$ &$E^{-\sigma,-\sigma^*}_N(u)$ &  order& $E^{-\sigma^*,-\sigma}_N(z)$ &    order &$E^{-\sigma^*,-\sigma}_N(q)$ & order & $E^{0,0}_N(q)$ & order &iter &CPU(s) \\
\midrule
\multirow{4}{*}{1.2}
 & 128       &  7.31e-05     &      *     &    7.98e-05     &     *     &   7.98e-05     &     *    &  3.19e-05     &     *            &  14     &     0.57    \\
     & 256       &  2.40e-05    &    1.61      &  2.66e-05    &    1.58       &  2.66e-05    &    1.58  & 9.11e-06    &    1.81        &  14     &     0.82    \\
     & 512       &  7.83e-06    &    1.62      &  8.82e-06    &    1.59       &  8.82e-06    &    1.59   &  2.59e-06    &    1.82          &  14     &     1.31    \\
     & 1024       &  2.54e-06    &    1.62      &  2.90e-06    &    1.60       &  2.90e-06    &    1.60    &  7.35e-07    &    1.82        &  14     &     3.78    \\
\hline
\multicolumn{2}{c}{Expected order}&  & 1.6 & &1.6 & & & & 1.6 & & \\
\hline

\multirow{4}{*}{1.4}

   & 128       &  7.92e-06     &      *     &    5.37e-06     &     *     &   5.37e-06     &     *     &   1.02e-06     &     * &   11     &     0.39    \\
    & 256       &  1.99e-06    &    1.99      &  1.36e-06    &    1.98       &  1.36e-06    &    1.98   &  1.98e-07    &    2.37       &  11     &     0.54    \\
    & 512       &  4.94e-07    &    2.01      &  3.38e-07    &    2.01       &  3.38e-07    &    2.01     &  3.67e-08    &    2.43  &  11     &     0.95    \\
    & 1024       &  1.22e-07    &    2.02      &  8.35e-08    &    2.02       &  8.35e-08    &    2.02     &  6.65e-09    &    2.47  &  11     &     2.77    \\
 \hline
\multicolumn{2}{c}{Expected order}&  & 2.0 & & 2.0& & & & 2.0& & \\
\hline

\multirow{4}{*}{1.6}

  & 128       &  1.45e-06     &      *     &    3.56e-07     &     *     &   3.56e-07     &     *    &   8.03e-08     &     *   &   11     &     0.30    \\
    & 256       &  2.77e-07    &    2.39      &  2.66e-05    &    2.64       &  5.71e-08    &    2.64   &  9.07e-09    &    3.15   &  11     &     0.43    \\
    & 512       &  5.28e-08    &    2.39      &  8.82e-06    &    2.61       &  9.38e-09    &    2.61     &  1.02e-09    &    3.16  &  11     &     0.76    \\
    & 1024       &  1.00e-08    &    2.40      &  2.90e-06    &    2.56       &  1.59e-09    &    2.56     &  1.15e-10    &    3.14 &  11     &     2.23    \\
 \hline
\multicolumn{2}{c}{Expected order}&  & 2.4 & &2.4 & & & &  2.4& & \\

\hline

\multirow{4}{*}{1.8}

  & 128       &  3.64e-07     &      *     &    1.75e-07     &     *     &   1.75e-07     &     *   &   3.45e-08     &     * &   11     &     0.29    \\
    & 256       &  5.08e-08    &    2.84      &  2.27e-08    &    2.95       &  2.27e-08    &    2.95     &  3.23e-09    &    3.42  &  11     &     0.43    \\
    & 512       &  7.09e-09    &    2.84      &  2.91e-09    &    2.96       &  2.91e-09    &    2.96      &  2.96e-10    &    3.45 &  11     &     0.74    \\
    & 1024       &  9.90e-10    &    2.84      &  3.72e-10    &    2.97       &  3.72e-10    &    2.97    &  2.69e-11    &    3.46   &  11     &     2.18    \\

 \hline
\multicolumn{2}{c}{Expected order}&  & 2.8 & & 2.8& & & &2.8 & & \\
 \bottomrule
\end{tabular}}
\end{table}
\end{subsection}
\end{section}
\begin{section}{Conclusion}
In this paper, we have investigated a spectral Petrov-Galerkin method for the optimal control problem governed by a two-sided space-fractional diffusion-advection-reaction equation. To compensate the weak singularities of
solutions near the boundary, we have analyzed the regularity of the fractional optimal control problem in weighted Sobolev space. If the regularity index of $f$ and $u_d$ is $r$, $r\geq0$, and $q \in L^2(\Omega)$, then the
regularity index of state $u$ and adjoint state  $z$ is $\min \{ r + \alpha , s \}$,  where $s = 2\alpha + \min\{\sigma,\sigma^* \}-1-\epsilon$,  $\epsilon>0$ is arbitrary small, $\sigma$ and  $\sigma^*$ are constants
defined as \eqref{sigma} depending on the fractional order $\alpha$ and $\theta$. Based on the regularity results,  we have provided error estimates of the spectral Petrov-Galerkin approximations to  $u$ in $L^2_{\omega^{-\sigma,-\sigma^*}}$-norm, $z$ in $L^2_{\omega^{-\sigma^*,-\sigma}}$-norm  and  $q$ in $L^2$-norm, where the convergence order is consistent with the obtained regularity index. 

A fast projected gradient algorithm with linear storage and quasi-linear complexity has been presented to solve the resulting discrete system efficiently. Numerical examples have verified the theoretical findings and shown the efficiency of the fast algorithm.  For the smooth source term $f$ and desired state $u_d$ in Example \ref{example1}, as well
$f$, $u_d$ with low regularity in Example \ref{example2}, numerical results have illustrated that our error estimate for state $u_N$ and adjoint state $z_N$ are optimal, while the convergence orders of control $q_N$ in standard
$L^2$-norm are better than the theoretical expectation.

An improved error estimate for the control $q_N$ and a better preconditioner which can further improve the performance (iterative numbers and CPU time) of the fast algorithm  can be considered in future work. Moreover, we are working at numerial methods for optimal control problems governed by time-fractional/space-fractional diffusion equations with additive noises.
\end{section}
\section*{Reference}

\bibliographystyle{plain}

 \end{document}